\documentclass[10pt]{article}

\usepackage{color}
\usepackage{graphicx}
\usepackage{amsmath,amsfonts,amssymb,graphics,amsthm}
\usepackage{hyperref}
\usepackage{tabularx}
\usepackage{enumerate}
\usepackage{bbm}
\usepackage[toc,page]{appendix}
\usepackage{epstopdf}
\usepackage{fullpage}
\usepackage{natbib}
\bibpunct{(}{)}{;}{a}{,}{,}

\numberwithin{equation}{section}

\newtheorem{Theorem}{Theorem}[section]
\newtheorem{Definition}[Theorem]{Definition}
\newtheorem{Lemma}[Theorem]{Lemma}
\newtheorem{Corollary}[Theorem]{Corollary}
\newtheorem{Proposition}[Theorem]{Proposition}

\newcommand{\R}{{\mathbb R}}

\newcommand{\CG}{{\mathcal G}}
\newcommand{\CH}{{\mathcal H}}

\newcommand{\noi}{\noindent}

\newcommand{\comment}[1]{}

\newcommand{\E}{{\bf E}}
\newcommand{\bP}{{\bf P}}
\newcommand{\px}{{\bf P}_x}
\newcommand{\pz}{{\bf P}_z}
\newcommand{\pxy}{{\bf P}_x^y}

\newcommand{\leb}[1]{\mathcal{L} (#1)}

\newcommand{\one}[1]{{\bf 1}(#1)}

\newcommand{\la}{\lambda}
\newcommand{\eps}{\epsilon}

\newcommand{\tpd}{ T_{\partial D}}
\newcommand{\wtd}{W_{\tpd}}
\newcommand{\larr}{\la[\rho,\rho']}
\newcommand{\tQ}{{\wt Q}}

\newcommand{\I}{\mathcal I}
\newcommand{\ol}{\overline}

\newcommand{\op}{\operatorname}

\newcommand{\ep}{\epsilon}
\newcommand{\rta}{\rightarrow}

\newcommand{\wt}{\widetilde}
\newcommand{\wh}{\widehat} 
\newcommand{\mcl}{\mathcal}

\newcommand{\N}{{\bf N}}
\newcommand{\Z}{{\bf Z}}

\begin{document}

\title{How round are the complementary components of planar Brownian motion?}
\author{Nina Holden\thanks{Holden's research was done during a visit to Microsoft Research, Redmond, WA.} \and \c Serban Nacu \and Yuval Peres \and Thomas S.\ Salisbury\thanks{Salisbury's research is supported in part by NSERC}}
\maketitle

\begin{abstract}
Consider a Brownian motion $W$ in ${\bf C}$ started from $0$ and run for time 1. Let $A(1), A(2),\dots$ denote the bounded connected components of ${\bf C}-W([0,1])$. Let $R(i)$ (resp.\ $r(i)$) denote the out-radius (resp.\ in-radius) of $A(i)$ for $i\in\bf N$. Our main result is that $\E[\sum_i R(i)^2|\log R(i)|^\theta ]<\infty$ for any $\theta<1$. We also prove that $\sum_i r(i)^2|\log r(i)|=\infty$ almost surely. These results have the interpretation that most of the components $A(i)$ have a rather regular or round shape.
\end{abstract}

\tableofcontents

\section{Introduction}

Consider a Brownian motion $W$ in ${\bf C}$, started from $0$ and killed at time 1. Let $A={\bf C} - W([0,1])$ denote the complement of the path of $W$, and let $A(1)$, $A(2)$, \dots be the connected components of $A$. See Figure \ref{fig-bm}. We are interested in the geometry of the sets $A(i)$, in particular in whether they are round and regular, or thin and irregular. We will show that, in a certain sense, the former is the case. Let $r(i)$ and $R(i)$ be the in-radius and out-radius of $A(i)$, respectively, that is, the radius of the largest disk contained in $A(i)$ and the radius of the smallest disk containing $A(i)$. Denoting the area of $A(i)$ by $\leb{A(i)}$, we have that
\begin{equation}
\pi r(i)^2\le \leb{A(i)} \le\pi R(i)^2.
\label{eq1}
\end{equation}
We think of $A(i)$ as being regular or round when these bounds are fairly good, and we will prove that this is indeed the case, in the sense that for $\theta\in\R$, a.s.,
\begin{equation}
\sum_{i=1}^\infty r(i)^{2}|\log r(i)|^\theta<\infty
\qquad\Leftrightarrow\qquad
\sum_{i=1}^\infty R(i)^{2}|\log R(i)|^\theta<\infty
\qquad\Leftrightarrow\qquad
\theta<1,
\label{eq41}
\end{equation}
and that the same result holds when we take the expectation of the left side in the first two inequalities.
The result \eqref{eq41}, and the analogous result with expectations, is immediate from Theorem \ref{thm2}, which is our main result, and Proposition \ref{prop5}.
\begin{Theorem}
With the above notation, for any $\theta<1$,
\begin{equation}
\E\left[\sum_{i=1}^\infty R(i)^2|\log R(i)|^\theta\right]<\infty.
\end{equation}
\label{thm2}
\end{Theorem}
\begin{Proposition}
With the above notation, almost surely
\begin{equation}
\begin{split}
\sum_{i=1}^\infty r(i)^2|\log r(i)|=\infty. 
\end{split}
\end{equation}
\label{prop5}
\end{Proposition}
\begin{figure}[ht]
\begin{center}
\includegraphics[scale=1]{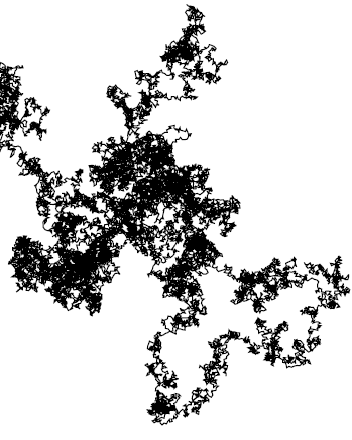}\quad
\includegraphics[scale=1]{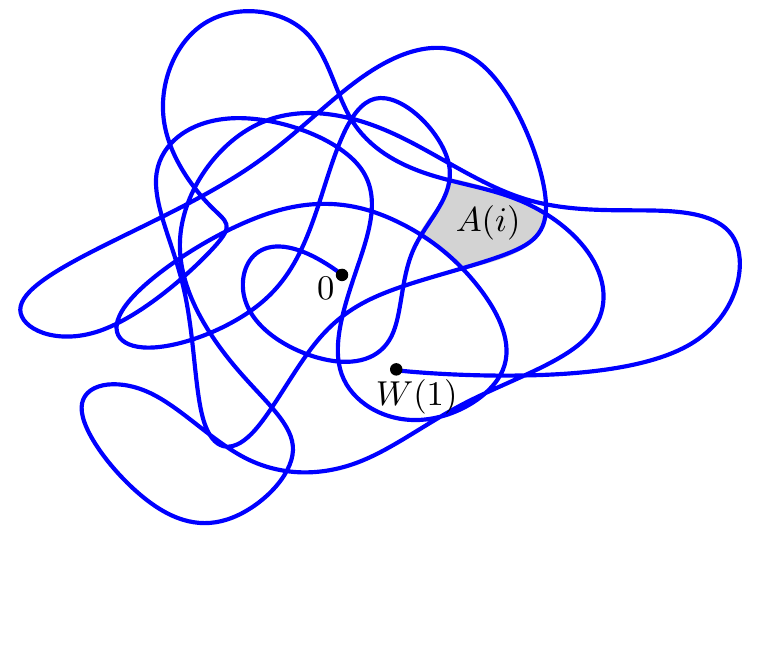}
\end{center}
\caption{Left: A planar Brownian motion (by Christopher J.\ Bishop). Right: A schematic drawing of a planar Brownian motion. The set $A(i)$ on the figure is one of the complementary connected components of the curve.
}
\label{fig-bm}
\end{figure}

Several authors have previously studied the properties of the connected components $A(i)$. \cite{mount} and \cite{legall} obtained results on the areas of the $A(i)$. In particular, Le Gall showed that the number $\mcl N(\eps)$ of components that have area at least $\eps$ satisfies $\lim_{\eps\rta 0} \eps |\log \eps|^2\mcl N(\eps) =2\pi $ (see also \cite{legall2} for additional details of the proof and slightly sharper estimates). \cite{wwthesis} derived a law ${\cal L}_1$ for the asymptotic shape of connected components, showed that it is related to the law of the infinite component of the Brownian loop, and proved that it can be computed from any one instance of a Brownian motion. \citet*{lsw5} 
used the SLE process to prove that the outer boundary of a Brownian loop has dimension 4/3. \cite{h14} studied the asymptotic number of connected components in the complement of a Wiener sausage in the plane. \cite{gh14} studied the shape of the complement of a Brownian motion on a torus of dimension $\geq 3$. Properties of the complement of a random walk have been studied by e.g.\ \cite{b08,ch96,wch97}.

The particular property of the sets $A(i)$ established in Theorem \ref{thm2} was motivated by a question raised by Chris Bishop in a conversation with the third author of this paper. \citet*{bishop97} 
proved that if $\Lambda$ is the limit set of an analytically finite Kleinian group, $\Omega(i)$ for $i\in\N$ is an enumeration of the components of $\Omega=S^2-\Lambda$, and $\wh R(i)$ is the diameter of $\Omega(i)$ for all $i\in\N$, then $\sum_{i} \wh R(i)^2<\infty$. He asked whether the same property holds for the complement of a Brownian motion, and we answer his question positively in Theorem \ref{thm2}.

The analogue of Theorem \ref{thm2} where $W$ is killed upon hitting the unit disk, rather than at time 1, also holds. The argument is available from the authors, upon request.

We will deduce Theorem \ref{thm2} from Proposition \ref{prop1}. Recall that the out-radius $R_z(\wt A)$ with center $z\in \bf C$ of a bounded set $\wt A\subset\bf C$ is defined by $R_z(\wt A):=\inf\{r\geq 0\,:\,\wt A\subset B(z,r) \}$, where $B(z,r)$ is the ball of center $z$ with radius $r$. In particular, the out-radius $R(\wt A)$ of $\wt A$ (with no specified center) is given by $R(\wt A)=\inf\{R_z(\wt A)\,:\,z\in \bf C\}$. If $\wt A$ is unbounded we define the out-radius of $\wt A$ relative to any point to be 0. 

\begin{Proposition}
Let $\tau$ be a unit rate exponential random variable independent of the planar Brownian motion $W$. For any $z\in\bf C$, let $R_z>0$ be the out-radius relative to $z$ of the component $A_z$ of ${\bf C}-W([0,\tau])$ containing $z$, which is a.s.\ well-defined for almost all $z\in\bf C$. Let $\leb{A_z}$ be the Lebesgue measure of $A_z$. Then for any $\theta<1$,
\begin{equation}
\int_{{\bf C}} \E\left[\frac{R_z^2 (|\log R_z| + 1)^\theta}{\mcl L(A_z)}  \right]\,dz<\infty.
\label{eq39}
\end{equation}
\label{prop1}
\end{Proposition}
We will prove Proposition \ref{prop1} in Section \ref{sec:prop2}, and we will prove Proposition~\ref{prop5} in Section \ref{sec:prop5}. We will explain at the very end of the introduction how to deduce Theorem \ref{thm2} from Proposition \ref{prop1}. The idea of the proof of Proposition \ref{prop1} is that, on the event that $A_z$ is bounded for some given $z\in\bf C$, the fraction $R_z^2 / \leb{A_z}$ is nearly scale invariant, so we can essentially condition on the value of $R_z$. We achieve this near scale invariance by counting the number of visits $N$ made by the Brownian motion to a neighborhood of $z$ of diameter approximately $R_z$ and conditioning on $N$. The restrictions of the Brownian motion to each of the $N$ visits are nearly independent, and in order for $\leb{A_z}$ to be small the restriction of the Brownian motion to two of the visits would need to trace each other closely, which is highly unlikely. In order to control the contribution from the logarithmic term in \eqref{eq39} we bound from above the probability that $W$ gets very close to a given point $z$.

Proposition \ref{prop5} will follow from a result of \cite{wwthesis} on the asymptotic law of the shape of $A(i)$, and a result of \cite{legall} on the areas $\leb{A(i)}$. Since $r(i)\leq R(i)$ it is immediate from Theorem \ref{thm2} that a version of the theorem with $r(i)$ instead of $R(i)$ also holds. In Section \ref{sec:prop5} we will, besides giving the proof of Proposition \ref{prop5}, give a short alternative proof of this result by using an estimate of \cite{legall}.

{\bf Notation.} 
We denote by $B(z,r)$ the ball with center $z$ and radius $r$. If $z=0$ we may write $B(r)$ instead of $B(0,r)$, and if $z=0$ and $r=1$ we may write $\bf D$ instead of $B(0,1)$. We denote by $d(x,S)$ the Euclidean distance between a point $x$ and a set $S$, and by $d(S,S')$ the distance between two sets $S$ and $S'$. For $S\subset\bf C$ we let $|S|$ denote the diameter of $S$. We denote by $T_S$ the hitting time of $S$ by a Brownian motion. We denote by $\px$ the law of Brownian motion started at $x$, and by $\pxy$ the law of Brownian motion started at $x$, conditioned to exit a domain $D$ at $y \in \partial D$, and killed upon exiting (it will be clear from the context what the domain $D$ is); the expectations $\E_x$ and $\E_x^y$ are defined similarly.

For $X_n,Y_n\in\R_+$, $n\in\N$, we write $X_n\preceq Y_n$ (resp.\ $X_n\succeq Y_n$) if there is a constant $C$ such that $X_n\leq C Y_n$ (resp.\ $X_n\geq CY_n$) for all $n\in\N$. We write $X_n\asymp Y_n$ if $X_n\preceq Y_n$ and $X_n\succeq Y_n$. We will always state explicitly if the constant $C$ depends on other constants; in other words, unless otherwise stated the constant $C$ will be assumed to be universal. An implicit constant in the proof of a result will be assumed to satisfy the same dependencies as the implicit constant in the statement of the result.

We end this section by explaining how to deduce Theorem \ref{thm2} from Proposition \ref{prop1}.
\begin{proof}[Proof of Theorem \ref{thm2} from Proposition \ref{prop1}]
	Let $\wt A(i)$ for $i\in\bf N$ be the bounded components of ${\bf C}-W([0,\tau])$, and let $\wt R(i)$ be the out-radius of $\wt A(i)$. Throughout the proof all implicit constants may depend on $\theta$. Since for any $i\in\N$ and $z\in\wt A(i)$ we have $\frac 12 R_z\leq \wt R(i)\leq R_z$,
	\[
	\sum_{i=1}^\infty \wt R(i)^2 (|\log \wt R(i)|+1)^\theta = \sum_{i=1}^\infty\frac{\wt R(i)^2(|\log \wt R(i)|+1)^\theta}{\leb{\wt A(i)}}\cdot \leb{\wt A(i)}
	\preceq \int_{\bf C} \frac{R_z^2(|\log R_z|+1)^\theta}{\leb{A_z}}\,dz.
	\]
	This estimate and Proposition \ref{prop1} imply
	\[
	\E\left[\sum_{i=1}^\infty \wt R(i)^2(|\log\wt R(i)|+1)^\theta\right] \preceq \int_{\bf C} \E \left[\frac{R_z^2(|\log R_z|+1)^\theta}{\leb{A_z}}\right]\,dz <\infty.
	\]
	By scale invariance of $W$ and independence of $W$ and $\tau$, 
	\begin{equation*}
	\begin{split}
	\E\left[ \sum_{i=1}^\infty \wt R(i)^2 (|\log\wt R(i)|+1)^\theta \right]
	&\geq \E\left[\one{1\leq \tau\leq 2} 
	\sum_{i=1}^\infty (\tau^{1/2} R(i))^2 (|\log( \tau^{1/2} R(i))|+1)^\theta\right] \\
	&\succeq \E\left[ \sum_{i=1}^\infty  R(i)^2 (|\log R(i)|+1)^\theta\right].
	\end{split}
	\end{equation*}
	We obtain the theorem by combining the latter two equations.
\end{proof}

\section{Estimates on the out-radius of Brownian components}
\label{sec:prop2}
In this section we will prove Proposition \ref{prop1}. Throughout the section $W$ is a planar Brownian motion, and $\tau$ is an independent unit rate exponential random variable. Unless otherwise stated we assume $W(0)=0$. 

The main input to the proof of the proposition is the following result. See Figure \ref{fig:prop3} for an illustration.
\begin{Proposition}
Let $\eta>0$, let $W$ be a planar Brownian motion started from $\partial B(\eta)$ or $\partial B(\eta/2)$, and let $\mcl G$ be the set of connected components of $\{\eta/2<|z|<\eta \}-W([0,\tau])$. Let $\mcl G^*$ be the set of $G\in\mcl G$ for which there is a path $\nu_G:[0,1]\to G$ satisfying $|\nu_G(0)|=\eta/2$ and $|\nu_G(1)|=\eta$. Then 
\begin{equation}
\E \left[\sum_{G \in \CG^*} \frac{\eta^2}{\leb{G}}\right] \preceq \frac{1}{|\log(\eta\wedge 1/2)|}.
\label{eq8}
\end{equation}
\label{prop3} 
\end{Proposition}

\begin{figure}[ht]
\begin{center}
\includegraphics[scale=1]{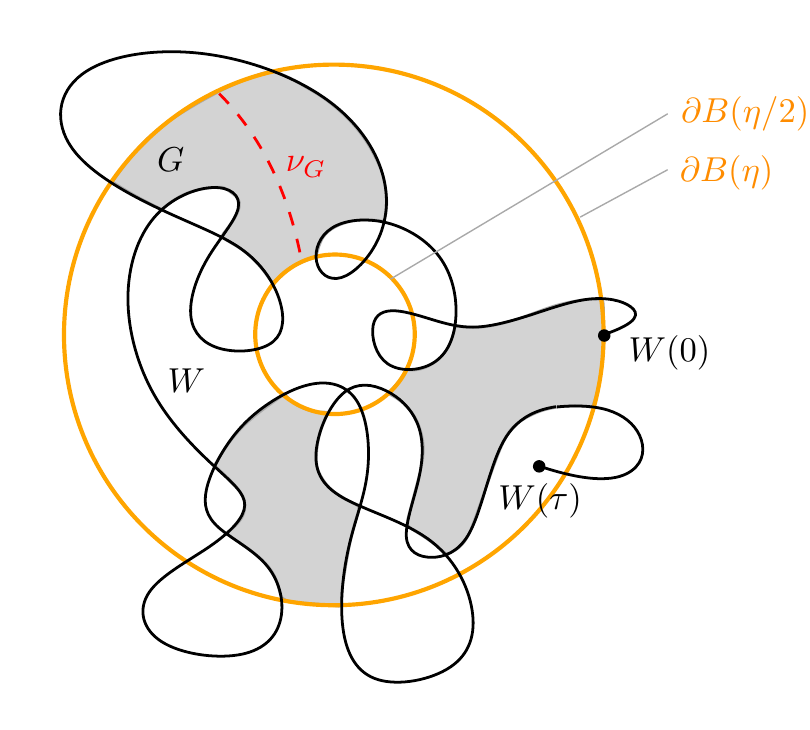}
\end{center}
\caption{Illustration of the statement of Proposition \ref{prop3}.
}
\label{fig:prop3}
\end{figure}

When proving Proposition \ref{prop3} we consider Brownian motions in $\lambda$-domains, which are defined as follows. See the left part of Figure \ref{fig-lambdadomain} for an illustration.

\begin{Definition}\label{def:lambdadomain}
A simply connected domain $D$ is called a $\la$-domain if there exist $0 < \eta_1 < \eta_2$ such that $\partial D$ is the union of the following sets:  
\begin{itemize}
\item a connected curve $\la$ which starts at $\partial B(\eta_1)$, ends at $\partial B(\eta_2)$, and is contained inside $\{z\,:\, 0 < \arg(z) < \pi,\,\eta_1\leq |z|\leq \eta_2 \} $, 
\item the segment $ \sigma = \{ -ri : \eta_1 \le r \le \eta_2 \}$, and
\item the two arcs on the circles $\partial B(\eta_1)$
and $\partial B(\eta_2)$ that connect $\la$ to $\sigma$, and which intersect the positive real line.
\end{itemize}
\end{Definition}

We will show in Lemma \ref{poly_bound} that for a collection of Brownian motions in a $\lambda$-domain the connected components of the complement of the Brownian motions that touch $\lambda$, have an expected total inverse area growing at most polynomially in the number of Brownian motions. This result will be deduced from Corollary \ref{cor_tail_est}, which follows from Proposition~\ref{exp_tail_for_la_domain} in the appendix. 

Before stating the corollary we need to introduce some notation. A {\bf polar rectangle} $q \equiv q(r, \theta, dr, d\theta)$
is a set of the form
$$ q = \{ r'\exp(i\theta') \, | \, r \le r' < r + dr, \, \theta \le \theta' 
< \theta + d\theta \}, $$
where $r, dr > 0$ and $\theta, d\theta \in [0, 2\pi)$. We call
$dr \equiv dr(q)$ the radial size, 
and $d\theta \equiv d\theta(q)$ the angular size. The boundary of
the rectangle is the union of two line segments and two circle arcs;
we refer to these as {\bf line sides} and {\bf arc sides}, respectively. 

Let $D$ be a $\la$-domain $D$ with parameters $\eta_1, \eta_2$, let $n \ge 1$, and let $\rho_1,\rho_2 \in (\eta_1, \eta_2)$ satisfy $\rho_1<\rho_2$. We will now define a collection ${\cal Q} (D, \rho_1, \rho_2, n)$ of polar rectangles. Let $dr = (\rho_2 - \rho_1) / n$ and $d\theta = \pi/(4n)$. For $j = 1, \ldots, n$, let
$$ \theta(j) = \inf \{ \arg(z) : z \in \la, |z| \in [\rho_1 + (j-1) \cdot dr, \rho_1 + j \cdot dr ] \} $$
and let
$$ Q_j = q( \rho_1 + (j-1) \cdot dr, \theta(j) - d\theta, dr, d\theta). $$
That is, we can obtain $Q_j$ as follows: start with a rectangle of size $dr \times d\theta$ that
lies just under the positive $x$-axis and touches the segment 
$[\rho_1 + (j-1) \cdot dr, \rho_1 + j \cdot dr ] $, and rotate it counterclockwise until it touches $\la$. We denote this collection of polar rectangles by ${\cal Q} (D, \rho_1, \rho_2, n)$ 
\begin{equation*}
{\cal Q} (D, \rho_1, \rho_2, n) = \{ Q_j\,:\, 1 \le j \le n\}.
\end{equation*}

\begin{figure}[ht]
\begin{center}
\includegraphics[scale=0.95]{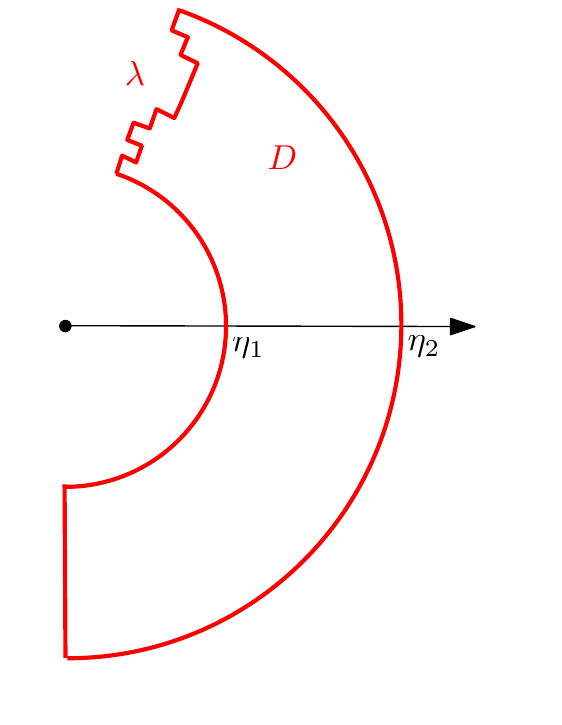}
\includegraphics[scale=0.95]{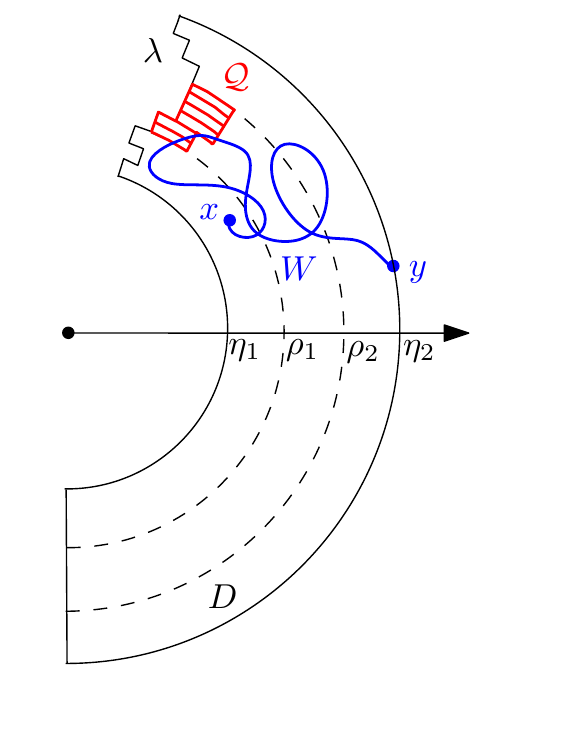}
\includegraphics[scale=0.95]{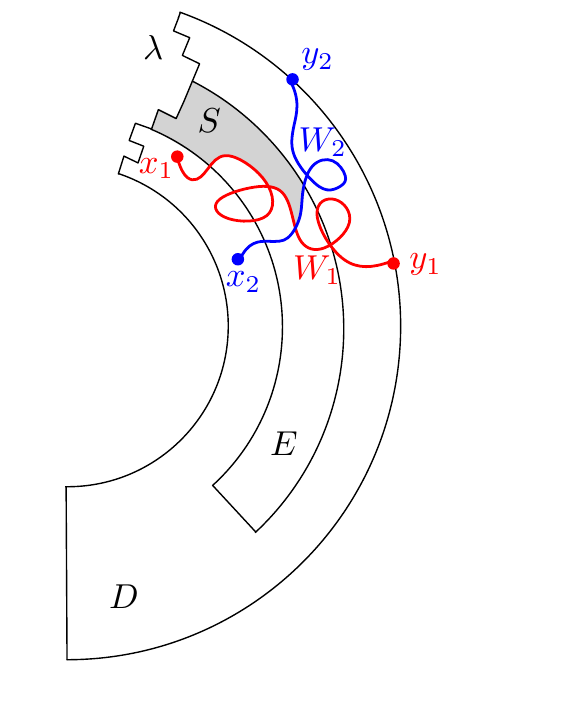}
\end{center}
\caption{
Left: The domain $D$ is a $\lambda$-domain with parameters $\eta_1$ and $\eta_2$, as defined in Definition \ref{def:lambdadomain}. 
Middle: By Corollary~\ref{cor_tail_est}, which follows from Proposition~\ref{exp_tail_for_la_domain}, the probability that the Brownian motion $W$ from $x$ to $y$ hits a large number of elements in $\mcl Q$ before exiting $D$ decays exponentially, uniformly in the number of elements of $\mcl Q$. Right: Illustration of Lemma \ref{poly_bound} for $n=2$. The lemma says that the expected inverse area of the gray region grows at most polynomially in the number of Brownian motions $W_i$. }
\label{fig-lambdadomain}
\end{figure}

The following corollary is immediate by Proposition~\ref{exp_tail_for_la_domain}. Observe that the collection ${\cal Q}(D,\rho_1, \rho_2, n)$ of polar rectangles may not satisfy the assumptions of the proposition, but the odd and even subcollections $\{ Q_{2j} \}$ and $\{ Q_{2j+1} \}$ each satisfy the assumptions, and we may apply the proposition with these collections instead. See the middle part of Figure \ref{fig-lambdadomain} for an illustration of the statement of the corollary.

\begin{Corollary}\label{cor_tail_est}
Let $D$ be a $\la$-domain with parameters $\eta_1, \eta_2$.
Let $\rho_1, \rho_2 \in (\eta_1, \eta_2)$ satisfy $\rho_1<\rho_2$, and let $\delta > 0$ and $n \ge 1$. Let $x \in D$ be such that either $\arg(x) \le -\pi / 4$, or $x$ can be connected to the positive real line by a path contained in $\{z\in D\,:\, |z|\not\in (\rho_1 - \delta ,\rho_2 + \delta) \}$. Let $y \in \partial D - \la$. Let $W$ be a Brownian motion started at $x$, conditioned to exit $D$ at $y$ and
killed upon exiting. Let $N$ be the number of rectangles in 
${\cal Q} (D, \rho_1, \rho_2, n)$ hit by $W$. Then
\begin{equation*}
\pxy(N \ge t) \preceq \mu^t \quad \forall t \ge 1,
\end{equation*}
where the implicit constant and $\mu<1$ are constants depending on $\eta_1, \eta_2, \rho_1, \rho_2$ and $\delta$.
\end{Corollary}

See the right part of Figure \ref{fig-lambdadomain} for an illustration of the statement of the following lemma.

\begin{Lemma}\label{poly_bound}
Let $D$ be a $\la$-domain with parameters $\eta_1, \eta_2$
and let $W_1, \ldots, W_n$ be independent Brownian motions, where each $W_i$ is started
at some point $x_i \in D$, conditioned to exit $D$ at some point $y_i \in \partial D - \la$,
and killed upon exiting. 
For $\rho_1, \rho_2 \in (\eta_1, \eta_2)$ satisfying $\rho_1<\rho_2$, let $E$ be the connected component of the following set which intersects the lower half-plane
$$ \{ z \in D : \rho_1 < |z| < \rho_2, \arg(z) > -\pi / 4 \}. $$

\noi Assume all points $x_i$ lie outside $E$ and can be connected to the positive real line by a path in $D-E$. Let $S$ be the union of all connected components of $E - \cup_i W_i([0,\infty))$
that touch $\la$; that is, the set of all points in $E$ that can be connected to $\la$
by a path that lies inside $E$ and do not meet the Brownian motions. Then
\begin{equation}\label{lebS}
\E \left[\leb{S }^{-1}\right] \preceq n^{3}
\end{equation}
where the implicit constant depends only on $\eta_1, \eta_2, \rho_1$ and $\rho_2$.
\end{Lemma}

\begin{proof}
Let $\delta = (\rho_2 - \rho_1) / 3$ and let $\rho_1^* = \rho_1 + \delta$, 
$\rho_2^* = \rho_2 - \delta$. For any $m \ge 1$, consider the collection of $m$ polar rectangles
${\cal Q} (D, \rho_1^*, \rho_2^*, m) = \{ Q_j \}$ defined right above the statement of Corollary~\ref{cor_tail_est}.
Each rectangle is contained inside $E$, touches $\la$, and has area at least 
$(\rho_1 \delta \pi / 4) / m^2$.

Let $N_i$ be the number of rectangles in ${\cal Q} (D, \rho_1^*, \rho_2^*, m)$ hit by $W_i$,
and let $N$ be the number of rectangles hit by some (at least one) $W_i$. If $N < m$, then
there is at least one $Q_j$ that is not hit by any of the Brownian motions. In that case $Q_j \subset S $,
so $ \leb{S }^{-1} \le \leb{Q_j}^{-1} \le C_1 m^2$, where $C_1 =  4/( \rho_1 \delta \pi ) $. Hence
$$ \bP[ \leb{S}^{-1} > C_1 m^2 ] \le \bP[N = m] \le \sum_{j=1}^n \bP[N_i \ge m / n]. $$

\noi Corollary~\ref{cor_tail_est} guarantees that all $N_i$ have exponential tails, so 
$$ \bP[ \leb{S}^{-1} > C_1 m^2 ] \preceq n \mu^{m / n}, $$
where $\mu < 1$ and the implicit constant may depend on $\eta_1, \eta_2, \rho_1^*, \rho_2^*$ and $\delta$, and we get the bound
$$ \E \left[ \leb{S}^{-1}\right] 
=\int_{0}^{\infty}\bP[ \leb{S}^{-1}>x ]\,dx
\preceq \int_{0}^{\infty}n\mu^{(x^{1/2} C_1^{-1/2}-1)/n}\,dx
= n^3 \int_{0}^{\infty} \mu^{y^{1/2} C_1^{-1/2}-1/n} \,dy
\preceq n^3.
$$
\end{proof}

We have obtained a polynomial upper bound for the expected inverse area induced by a collection of Brownian motions. In Lemma \ref{prop:UV} we will obtain an exponential bound for a related expectation.

Given $\eta_2>\eta_1>0$ define 
$$\rho_1 = (2/3) \eta_1 + (1/3) \eta_2, \quad \rho_2 = (1/3) \eta_1 + (2/3) \eta_2,$$ 
and let $V=V(\eta_1,\eta_2)$ and $U=U(\eta_1,\eta_2)$ be the following sets satisfying $U\subset V\subset \bf C$. Let $V$ be an annulus with the negative $y$-axis removed
\begin{equation}
V =V(\eta_1,\eta_2)= \{z : \eta_1<|z|<\eta_2 \} - \{ z : \arg(z) = - \pi/2 \},
\label{eq27}
\end{equation} 
and let $U$ be a smaller set of a similar shape, except that part of the set close to the negative $y$-axis has been removed
\begin{equation}
U =U(\eta_1,\eta_2)= \{ z : \rho_1 < |z| < \rho_2\}  - \{ z : |\arg(z)+\pi/2| < \pi/4 \}.
\label{eq25}
\end{equation} 
See Figure \ref{fig-UV-lem-la} for an illustration of the sets $U$ and $V$ and of the statement of the following lemma.
\begin{Lemma}
Given $\eta_2>\eta_1>0$ let $U=U(\eta_1,\eta_2)$ and $V=V(\eta_1,\eta_2)$ be defined as above. For $n\in\N$ and $x_i\in\partial U$ and $y_i\in \partial V$ for $i\in\{1,\dots,n\}$ let $W_i$ be a Brownian motion started from $x_i$ conditioned to exit $V$ at $y_i$, and killed upon exiting $V$. For $i\in\{0,\dots,n\}$, $y_0\not\in V$ and $x_{n+1}\not\in U$ let $\wh W_i$ be a path contained in ${\bf C}-U$ which starts at $y_{i}$ and ends at $x_{i+1}$. Let $\mcl G$ be the connected components of $V-\cup_i W_i([0,\infty))-\cup_i \wh W_i([0,\infty))$, and let $\mcl G^*\subset\mcl G$ be the set of $G\in\mcl G$ for which there exists a path $\nu_G:[0,1]\to G\cap \bf H$ satisfying $|\nu_G(0)|=\eta_1$ and $|\nu_G(1)|=\eta_2$, where $\bf H$ is the upper half-plane. Then we have 
\begin{equation}\label{eqmainexpbound2}
\E \left[\sum_{G \in \CG^*} \frac{\eta_2^2}{\leb{G\cap U}}\right] \preceq \mu^n,
\end{equation}
where the sum is defined to be 0 if $\CG^*$ is empty. The implicit constant and $\mu\in(0,1)$ depend only on the ratio $\eta_2/\eta_1$. 
\label{prop:UV}
\end{Lemma}

\begin{proof}
For any $m\in\bf N$ we can partition the annular region $V\cap \bf H$ into $m^2$ polar rectangles of size $(\eta_2-\eta_1)/m\times \pi/m$. The union of the boundaries of those rectangles forms a polar grid. Let $LP(m)$ be the set of all lattice paths on the polar grid that connect the circles $\partial B(\eta_1)$ and $\partial B(\eta_2)$, lie in $V\cap\bf H$, and do not self-intersect. Define an order relation on $LP(m)$ by saying that $\lambda>\lambda'$ if $\lambda$ is left of $\lambda'$; more precisely, we say that $\lambda>\lambda'$ if $\lambda\neq\lambda'$ and there are no parts of $\lambda$ strictly on the right side of $\lambda'$. With this order relation, for each $G\in\CG^*$ define $\lambda_m(G):=\max\{\lambda\in LP(m)\,:\,\lambda\subset G \}$; if the set $\{\lambda\in LP(m)\,:\,\lambda\subset G \}$ is empty then $\lambda_m(G)$ is not defined, but observe that a.s.\ for any $G\in\mcl G^*$ the path $\lambda_m(G)$ is well-defined for all sufficiently large $m$. The curve $\lambda_m(G)$ can be viewed as an approximation to the left boundary of $G$ from the right, with the additional constraint that it should be contained in $\bf H$.

For any $\lambda\in LP(m)$ let $G_\lambda$ be the union of all components of $U-\cup_i W_i([0,\infty))-\cup_i \wh W_i([0,\infty))$ which are on the right side of $\lambda$ and whose closures have a non-empty intersection with $\lambda$. In other words, $z\in G_\lambda$ if $z$ can be connected to $\lambda$ by a path in $U$ that does not intersect any of the $W_i$ or $\wh W_i$, and which is on the right side of $\lambda$. See Figure \ref{fig-UV-lem-la}. Observe that if $\lambda=\lambda_m(G)$ for some $G\in\mcl G^*$ then we have $G_\lambda\subset G\cap U$. (However, the inverse inclusion does not necessarily hold.)

\begin{figure}[ht]
\begin{center}
\includegraphics[scale=1]{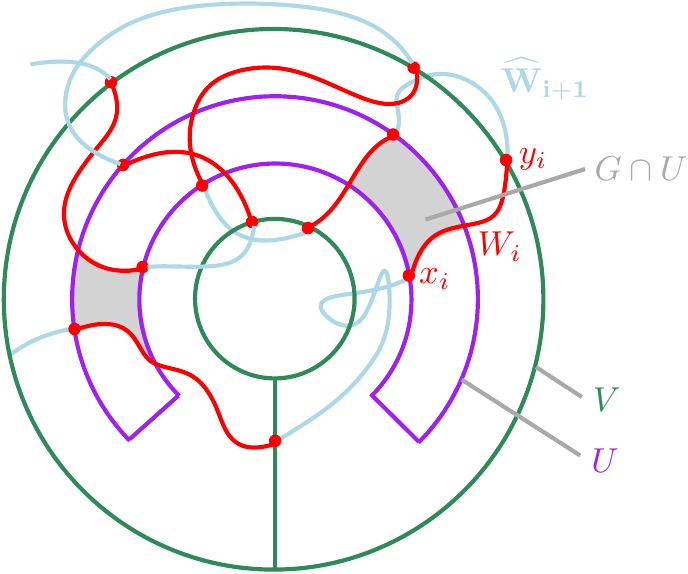}\qquad
\includegraphics[scale=1]{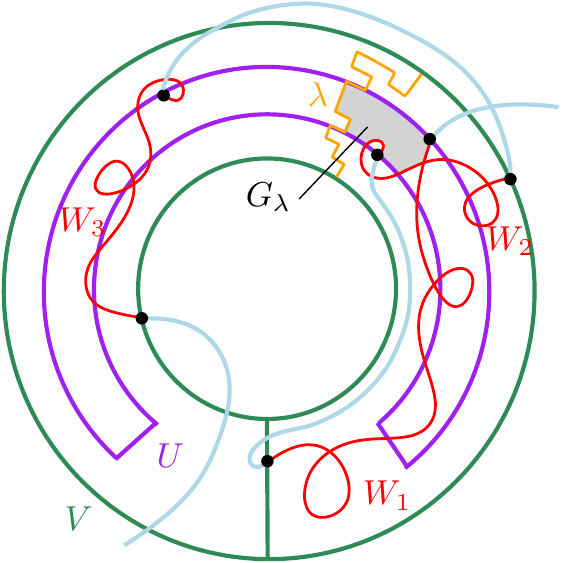}
\end{center}
\caption{Left: Illustration of Lemmas \ref{prop:UV} and \ref{prop2}. Right: Illustration of $G_\lambda$ in the proof of Lemma \ref{prop:UV}.}
\label{fig-UV-lem-la}
\end{figure}

Note next that $\CG^*$ is finite; in fact $|\CG^*| \le n+1$. Between two curves $\nu_G$ and $\nu_{G'}$ for $G,G'\in\mcl G^*$, there must be at least one Brownian motion $W_i$, since none of the $W_i$ cross these curves. Hence $|\CG^*| \le n+1$.

Since $\CG^*$ is finite, for all $m$ sufficiently large the left boundary approximation $ \la_m(G) $ is well-defined for all $G \in \CG^*$. Let $\lambda_m(\mcl G^*)=\{ \lambda_m(G)\,:\, G\in\mcl G^* \}$. By our observation $G_\lambda\subset G\cap U$, 
$$ \sum_{G \in \CG^*} \frac{\eta^2_2}{\leb{G \cap U}}  \leq 
\liminf_{m \rightarrow \infty} \sum_{\la \in \lambda_m(\mcl G^*)} \frac{\eta^2_2}{\leb{G_\lambda}}.$$

Fatou's lemma now gives
\begin{eqnarray*}
\E\left[ \sum_{G \in \CG^*} \frac{\eta^2_2}{\leb{G \cap U}}\right] 
& \le  
& \liminf_{m \rightarrow \infty} \sum_{\la \in LP(m)} \E \left[\frac{\eta^2_2}{\leb{G_\lambda}} {\bf 1} (\la \in \lambda_m(\mcl G^*))\right]  \\
& = & \liminf_{m \rightarrow \infty} 
\sum_{\la \in LP(m)} \E\left[ \frac{\eta^2_2}{\leb{G_\lambda}} \,|\, \la \in \lambda_m(\mcl G^*)\right] \cdot \bP[\la \in \lambda_m(\mcl G^*)].
\end{eqnarray*}

By our definition of $V$, on the event that $\la \in \lambda_m(\mcl G^*)$ each of the paths $W_i$ are either to the left of $\la$ or to the right of $\la$; an upcrossing $W_i$ cannot be both to the left and to the right of $\la$, since this would imply that it crosses either $\la$ or the negative imaginary axis. The event $ \{ \la \in \lambda_m(\mcl G^*) \} $ amounts to two things. First, none of the paths $W_i$ or $\wh W_i$ hit $\la$. Second, every rectangle in the polar lattice of size $m$ that lies just left of $\la$ and is inside the upper half-plane is hit by at least one of the paths $W_i$ or $\wh W_i$. The latter condition introduces a rather complicated dependency among the Brownian motions $W_i$ that start to the left of $\lambda$ and the paths $\wh W_i$, but has no effect on the Brownian motions $W_i$ started to the right of $\lambda$. These remain independent even after conditioning on $ \{ \la \in \lambda_m( \mcl G^*) \} $. Observe that if $\lambda\in\lambda_m(\mcl G^*)$ and $x_i$ is on the right side of $\lambda$, then $x_i$ may be connected to the positive real line by a path staying inside $V-U$, since the path $\wh W_{i-1}$ connects $x_i$ to $V^c$ by a path in $V-U$. By Lemma \ref{poly_bound},
\begin{equation*}
\E \left[\frac{\eta^2_2}{\leb{G_\lambda}} \, | \la \in \lambda_m(\mcl G^*) \right] \preceq n^3.
\end{equation*}
Therefore, using $|\mcl G^*|\leq n+1$,
\begin{equation}
\E \left[ \sum_{G \in \CG^*} \frac{\eta^2_2}{\leb{G \cap U}} \right]  \preceq n^3 \liminf_{m \rightarrow \infty} \sum_{\la \in LP(m)} {\bP}\left[\la \in \lambda_m(\mcl G^*)\right] 
= n^3 {\E}[ |\mcl G^*|] 
 \preceq n^4 {\bP}[\mcl G^* \ne \emptyset].
 \label{eq29}
\end{equation}
For each $i$ there is a uniformly positive probability (depending on the ratio $\eta_2/\eta_1$) that $W_i$ intersects both the negative real line and the positive real line. Since this event implies that $\CG^*=\emptyset$, we see that 
$\bP[\CG^* \ne \emptyset]$ decays exponentially with $n$, so the right side of \eqref{eq29} is $\preceq \mu^n$, where $\mu\in(0,1)$ depends only on the ratio $\eta_2/\eta_1$.
\end{proof}

The following is an almost immediate corollary of Lemma \ref{prop:UV}, and will imply the upper bound of $\preceq 1$ in Proposition \ref{prop3}.
\begin{Lemma}
Let $\eta_1,\eta_2,U,V$ be as in Lemma \ref{prop:UV}. For $i\in\bf N$ let $x_i\in\partial U$, $y_i\in\partial V$ and
 $W_i$ be as in Lemma \ref{prop:UV}. For each $n\in\bf N$ and 
 $i\in\{0,\dots,n\}$ let $\wh W_i^n$ satisfy the same properties as $\wh W_i$ in Lemma \ref{prop:UV} (but where the terminal point of $\wh W_n^n$ is not required to be equal to $x_{n+1}$). For each $n\in\bf N$ let $\mcl G_n^*$ be defined as the set $\mcl G^*$ in Lemma \ref{prop:UV}, i.e., it is the set of components $G$ of 
 $V-\cup_{1\leq i\leq n} W_i([0,\infty))-\cup_{0\leq i\leq n} \wh W^n_i([0,\infty))$ for which there exists a path 
 $\nu_G:[0,1]\to G\cap \bf H$ satisfying $|\nu_G(0)|=\eta_1$ and $|\nu_G(1)|=\eta_2$. Then
\begin{equation*}
\E \left[\sum_{n=1}^\infty\sum_{G \in \CG^*_n} \frac{\eta_2^2}{\leb{G\cap U}}\right] \preceq 1,
\end{equation*}
with the implicit constant depending only on the ratio $\eta_2/\eta_1$.
\label{prop7}
\end{Lemma}

\begin{proof}
By Lemma \ref{prop:UV}, for all $n\in\bf N$,
\begin{equation*}
\E\left[\sum_{G \in \CG^*_n} \frac{\eta^2_2}{\leb{G\cap U}}\right] \preceq \mu^n,
\end{equation*}
and the lemma follows by summation over $n\in\N$.
\end{proof}

In our proof of Proposition \ref{prop3} we will need a variant of Lemma \ref{prop:UV} where the law of the Brownian motions $W_i$ are slightly different. The following lemma gives us the Radon-Nikodym of the process of interest with respect to the one considered in Lemma \ref{prop:UV}. See Figure \ref{fig-UV-lem-la}.

\begin{Lemma}
Let $U\subset V\subset\bf C$ be open sets such that $\op{dist}(U,V^c)>0$. Let $W$ be a Brownian motion satisfying $W(0)=y_0\not\in U$, and define for $i\in\N$
\begin{equation}
\begin{split}
&T_0=0,\qquad 
T_{2i-1}= \inf\{t\geq T_{2i-2}\,:\,W(t)\in U \},\qquad
T_{2i} = \inf\{t\geq T_{2i-1}\,:\,W(t)\in V^c \}\\
&x_i = W(T_{2i-1}),\qquad
y_i = W(T_{2i}),\qquad t_i = T_{i}-T_{i-1},\\
&W_i(t) = W\big((t\wedge t_{2i})+T_{2i-1}\big),\qquad
\wh W_i(t) = W\big((t\wedge t_{2i+1})+T_{2i}\big),\qquad
\forall t\geq 0.
\end{split}
\label{eq28}
\end{equation}
Let $\mcl F$ be the $\sigma$-algebra generated by $\wh W_i$ and  $\one{\tau\leq T_i}$ for $i\in\N$. Let $\ol W_i$ be a Brownian motion started from $x_i\in\partial U$, conditioned to exit $V$ at $y_i\in\partial V$, and killed upon exiting $V$. Let $\ol t_{2i}$ be the duration of $\ol W_i$, and let $f_{x_i,y_i}$ be the probability density function of $\ol t_{2i}$. Conditioned on $\mcl F$, and on the event that $\tau\in(T_{2n},T_{2n+1})$ for some $n\in{\bf N}$, the upcrossings $W_i$ are independent for $i\in\{1,\dots,n\}$, and each upcrossing $W_i$ for $i\in\{1,\dots,n\}$ has a law which is absolutely continuous with respect to the law of $\ol W_i$, such that the Radon-Nikodym derivative of the law of $W_i$ with respect to the law of $\ol W_i$ is given by 
\begin{equation}
\frac{\exp(-t_{2i})}{\int_{0}^{\infty} f_{x_i,y_i}(s)\exp(-s)\,ds}.
\label{eq:rn}
\end{equation}
\label{prop2}
\end{Lemma}

\begin{proof}
Conditioned only on the points $x_i,y_i$, the upcrossings $W_i$ are independent, and $W_i$ has the same law as $\ol W_i$ for each $i\in\bf N$. On the event we consider in the lemma we have $\tau\in(T_{2n},T_{2n+1})$, which affects the law of the duration of the upcrossings, so to conclude the proof of the lemma we need to describe the joint law of $t_{2i}$ for $i\in\{1,\dots,n\}$ and $n\in\bf N$, conditioned on $\mcl F$ and on the event $\tau\in(T_{2n},T_{2n+1})$. For $y\in U^c$ and $x\in\partial U$ let $\wt f_{y,x}$ denote the probability density function for the duration of a Brownian motion started from $y$, conditioned to exit $U^c$ at $x$, and killed upon exiting $U^c$. Let $i\in\N$ and $n\geq i$, and for $j\in\{1,\dots,2n+1 \}$ let $A_j\subset\R_+$ be a set of positive measure. By Bayes' law, the conditional law of $t_{2i}$ given that $\tau\in(T_{2n},T_{2n+1})$ and that $t_j\in A_j$ for all $j\in\{1,\dots,2n+1\}-\{2i \}$, is characterized by
\begin{equation*}
\begin{split}
&\bP[ t_{2i}\in A_{2i}\,|\,\tau\in(T_{2n},T_{2n+1});\,t_j\in A_j\text{\,\,for\,\,}j\in\{1,\dots,2n+1\}-\{2i\}]\\
&\qquad= \frac{\bP[\tau\in(T_{2n},T_{2n+1});\,t_j\in A_j\text{\,\,for\,\,}j\in \{1,\dots,2n+1\} ]}
{\bP[\tau\in(T_{2n},T_{2n+1});\,t_j\in A_j\text{\,\,for\,\,}j\in\{1,\dots,2n+1\}-\{2i\}]} \\
&\qquad=\frac{\int_{A_{1}}\dots \int_{A_{2n+1}} f(s_1,\dots,s_{2n+1}) \cdot g\left( \sum_{j=1}^{2n} s_j, \sum_{j=1}^{2n+1} s_j\right)\,ds_1\dots ds_{2n+1}} 
{\qquad \int_{\wt A_{1}}\dots \int_{\wt A_{2n+1}} f(s_1,\dots,s_{2n+1}) \cdot g\left( \sum_{j=1}^{2n} s_j, \sum_{j=1}^{2n+1} s_j\right)\,ds_1\dots ds_{2n+1}}
\\
&\qquad= \frac{\int_{A_{2i}} f_{x_i,y_i}(s) \exp(-s)\,ds}{\int_0^\infty f_{x_i,y_i}(s) \exp(-s)\,ds},
\end{split}
\end{equation*}
where 
\begin{equation*}
\begin{split}
&f(s_1,\dots,s_{2n+1}) := \wt f_{y_0,x_1}(s_1) \cdot f_{x_1,y_1}(s_2) \cdot \dots \cdot \wt f_{y_{n},x_{n+1}}(s_{2n+1}),\quad s_i\geq 0,\\
&g(\wt s_1,\wt s_2) := \bP[ \wt s_1\leq \tau< \wt s_2 ] = \exp(-\wt s_1)(1-\exp(\wt s_2-\wt s_1)),\quad 0\leq \wt s_1\leq \wt s_2,
\end{split}
\end{equation*}
$\wt A_j=A_j$ for $j\neq 2i$, and $\wt A_{2i}=\R_+$.
We observe that $t_{2i}$ is independent of the values of $t_j$ for $j\in\{1,\dots,2n+1\}-\{2i \}$, and that the Radon-Nikodym derivative of the law of $t_{2i}$ with respect to the law of $\ol t_{2i}$ is given by \eqref{eq:rn}. This implies the lemma.
\end{proof}

\begin{Definition}
In the setting of Lemma \ref{prop2}, we say that $W_i$ is the {\bf $i$th upcrossing of $W$ from $\partial U$ to $\partial V$}.
\end{Definition}

The following lemma says that if we rescale the sets $U,V$ in the above lemma such that they have a diameter of order $\eta\in(0,1)$, then the supremum of the considered Radon-Nikodym derivative converges uniformly to 1 as $\eta\rta 0$. Furthermore, the lemma gives a lower bound on the number of upcrossings before time $\tau$ for small $\eta$.

\begin{Lemma}
Let $U\subset V\subset\bf C$ be as in Lemma \ref{prop2}, and assume $V$ is bounded. Let $\eta\in(0,1/2)$, and define $U^\eta:=\{\eta z\,:\,z\in U \}$ and $V^\eta:=\{\eta z\,:\,z\in V \}$. For $x\in V^\eta$ and $y\in\partial V^\eta$ let $f^\eta_{x,y}$ be the probability density function for the duration of a Brownian motion started at $x$, conditioned to exit $V^\eta$ at $y$, and killed upon exiting $V^\eta$. Then 
\begin{equation}
\lim_{\eta\rta 0} c_\eta = 1\qquad\text{for}\qquad
c_\eta:=\sup_{x\in\partial U^\eta,y\in\partial V^\eta} \left(\int_{0}^{\infty} f_{x,y}^\eta(s)\exp(-s)\,ds \right)^{-1}.
\label{eq23}
\end{equation}
In other words, if $\mcl F$, $\ol W_i$ and the objects in \eqref{eq28} are defined as in Lemma \ref{prop2}, but with $U^\eta$ and $V^\eta$ instead of $U$ and $V$, then, conditioned on $\mcl F$ and on the event $\tau\in(T_{2n},T_{2n+1})$ for some $n\in\N$, the supremum of the Radon-Nikodym derivative of the law of $W_i$ with respect to the law of $\ol W_i$ for $i\leq n$, converges uniformly to 1 as $\eta\rta 0$. 

Furthermore, if there is a constant $C>0$ such that $d(W(0),V^\eta)\leq C\eta$ then $\bP[\tau\in(T_{2n},T_{2n+1}) ]\preceq |\log\eta|^{-1}$ for all $n\in\N\cup\{0 \}$, where the implicit constant may depend on $U$, $V$ and $C$.
\label{prop6}
\end{Lemma}
\begin{proof}
Let $t_{x,y}$ be a random variable with the law of the duration of a Brownian motion started at $x\in V^\eta$, conditioned to leave $V^\eta$ at $y\in\partial V^\eta$, and killed upon exiting $V^\eta$. \cite{cranston-mcconnell} proved that  $\E[t_{x,y}]\preceq \leb{V^\eta} \asymp \eta^2$, so $\bP[t_{x,y}>t]\preceq \eta^2/t$. The estimate \eqref{eq23} follows, since for some constant $c>0$
\begin{equation*}
\int_{0}^{\infty} f_{x,y}^\eta(t)\exp(-t)\,dt
\geq \bP[t_{x,y}\leq\eta]\cdot\exp(-\eta) \geq (1-c\eta)\exp(-\eta)\rta 1\qquad\text{as\,\,}\eta\rta 0.
\end{equation*}

To prove the second claim of the lemma, we assume without loss of generality that $0\in U$, and we let $r\asymp 1$ be such that $B(r)\subset U$ and $d(B(r),V^c)>2r$. Observe that if $\wt\tau=\inf\{t\geq 0\,:\,W(t)\in B(\eta r) \}$ then
\begin{equation*}
\bP[\tau\in(T_{2n},T_{2n+1})]
\leq \sup_{z\in\partial V^\eta} \bP_z[ \tau<\wt\tau ].
\end{equation*}
For any $z\in\partial V^\eta$,
\begin{equation}
\bP_z[ \tau>\wt\tau ] 
= \frac{\E_z\left[\int_{0}^{\infty} \one{W(t)\in B(\eta r);t\leq\tau}\,dt \right]}{\E_z\left[\int_{0}^{\infty} \one{W(t)\in B(\eta r);t\leq\tau}\,dt\,|\,\tau>\wt\tau \right]}
= \frac{ \int_{0}^{\infty} \int_{B(\eta r)} \exp(-t) t^{-1} \exp(-|z-w|^2/(2t))\,dw\,dt }
{ \int_{0}^{\infty} \int_{B(\eta r)} \exp(-t) t^{-1} \exp(-|\eta r-w|^2/(2t))\,dw\,dt },
\label{eq32}
\end{equation}
where the estimate for the denominator is obtained by considering the expected time spent in $B(\eta r)$ by a Brownian motion started at $\partial B(\eta r)$ and stopped at time $\tau$. 

Since for any $z\in\partial V^\eta$ and $w\in B(\eta r)$
\begin{equation*}
\begin{split}
\exp(-|\eta r&-w|^2/(2t)) - \exp(-|z-w|^2/(2t))\\
&=\exp(-|\eta r-w|^2/(2t)) \left( 
1- \exp(-|z-w|^2/(2t)+|\eta r-w|^2/(2t))\right)\\
&\preceq \exp(-|\eta r-w|^2/(2t)) \left( 
|z-w|^2/(2t)-|\eta r-w|^2/(2t)
\right)\preceq \exp(-|\eta r-w|^2/(2t)) \eta^2/t,
\end{split}
\end{equation*}
it follows that
\begin{equation}
\begin{split}
\bP_z[ \tau<\wt\tau ] 
= 1-\bP_z[ \tau>\wt\tau ] 
\preceq \frac{ \int_{0}^{\infty} \int_{B(\eta r)} \exp(-t) t^{-1} \exp(-|\eta r-w|^2/(2t)) \cdot \eta^2/t
\,dw\,dt }
{ \int_{0}^{\infty} \int_{B(\eta r)} \exp(-t) t^{-1} \exp(-|\eta r-w|^2/(2t))\,dw\,dt }.
\end{split}
\label{eq37}
\end{equation}
The numerator of \eqref{eq37} is
\begin{equation*}
\preceq \int_{0}^{\eta^2}dt
+ \int_{\eta^2}^{1} \eta^2\cdot \eta^2/t^2\,dt
+ \int_{1}^{\infty} \eta^2\exp(-t) \cdot \eta^2/t^2\,dt
\preceq \eta^2, 
\end{equation*}
and the denominator of \eqref{eq37} is
\begin{equation*}
\succeq \int_{\eta^2}^{1} \eta^2 t^{-1}\,dt\asymp \eta^2|\log\eta|.
\end{equation*}
We conclude that $\bP_z[ \tau<\wt\tau ] \preceq |\log\eta|^{-1}$.
\end{proof}

We are now ready to prove Proposition \ref{prop3}.
\begin{proof}[Proof of Proposition \ref{prop3}]
\noindent {\bf Step 1.} Define $D:=\{ \eta/2<|z|<\eta \}$. We will now define two disjoint annuli $\wt D=\{\wt\eta_1<|z|<\wt \eta_2 \}$ and $\wt D'=\{\wt\eta'_1<|z|<\wt \eta'_2 \}$ contained in $D$, where $\eta/2<\wt\eta'_1<\wt \eta'_2<\wt\eta_1<\wt \eta_2<\eta$. First let $\wh D$ and $\wh D'$ be two arbitrary disjoint annuli centered at 0 and contained in $D$. Letting $Q$ denote the open first quadrant let $\phi:\wh D\cap Q\to \wt V$ be a conformal map with $\wt V=V(\wh\eta,\eta)$ for some $\wh{\eta}<\eta$ (recall \eqref{eq27}), such that $\phi$ maps the restriction of $\partial (\wh D\cap Q)$ to the real (resp.\ imaginary) axis to $\partial B(\eta)$ (resp.\ $\partial B(\wh\eta)$). Appropriate $\wh\eta$, $\wt V$ and $\phi$ exist since the extremal distance between $\partial B(\eta)$ and $\partial B(\wh{\eta})$ in $\wt V$ (as defined in e.g.\ \cite[Chapter 3.7]{lawler-book}) varies continuously as we vary $\wh\eta$, and it converges to 0 (resp.\ $\infty$) when $\wh\eta\rta \eta$ (resp.\ $\wh\eta\rta 0$). Therefore we can find an appropriate $\wh\eta$ such that the extremal distance between $\partial B(\wh\eta)$ and $\partial B(\eta)$ in $\wt V$ is the same as the extremal distance between the intersection of $\partial (\wh D\cap Q)$ with the real and imaginary axis, respectively, and existence of an appropriate conformal map $\phi$ follows. Furthermore, we may assume that $|\phi'|\asymp 1$. Observe that $\phi$ interchanges the radial and angular coordinates, in the sense that for any $z\in\wh D\cap Q$ the radial (resp.\ angular) coordinate of $\phi(z)$ is determined by the angular (resp.\ radial) coordinate of $z$. Define $\wt D$, $\wt{\eta}_1$ and
 $\wt \eta_2$ such that $\phi(\wt D\cap 
 Q)={\bf H}\cap \wt V$, with $\bf H$ denoting 
 the upper half-plane. Define $\wt D',\wt{\eta}'_1,\wt{\eta}'_2$ in 
 exactly the same way, except that we consider a 
 conformal map from $\wt{\phi}:\wh D'\cap Q\to \wt V'$, where 
 $\wt V'=V(\wh\eta',\eta)$ for some $\wh{\eta}'<\eta$.

\begin{figure}[ht]
\begin{center}
\includegraphics[scale=1]{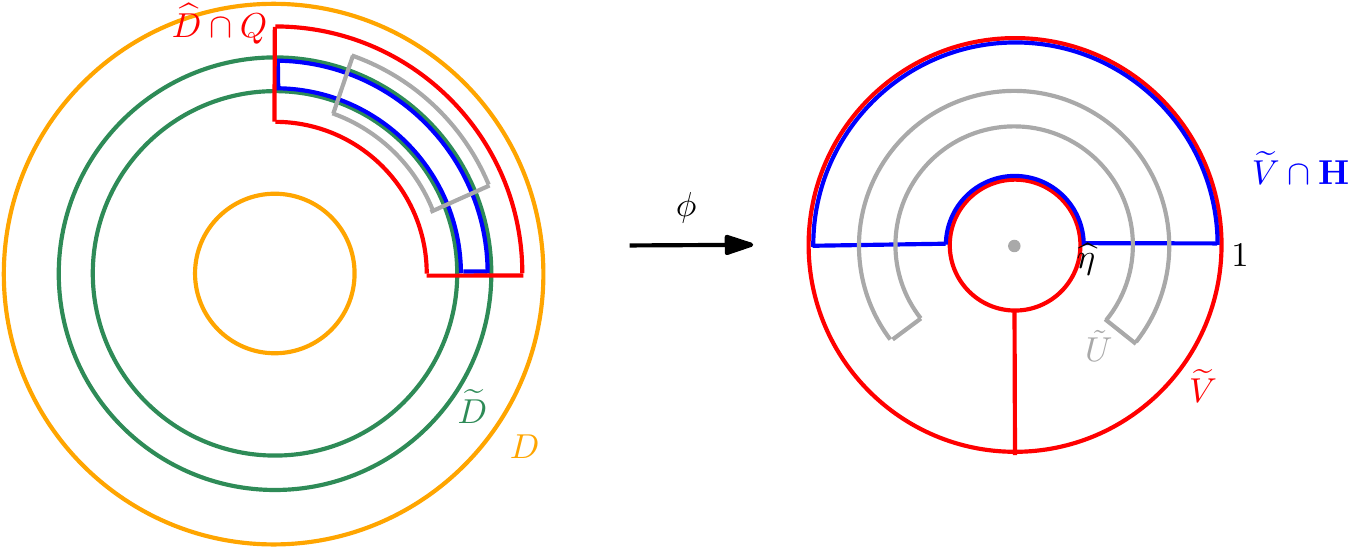}
\end{center}
\caption{Illustration of objects defined in step 1 in the proof of Proposition \ref{prop3}.
}
\label{Dtilde3}
\end{figure}

Since $\wt D$ and $\wt D'$ are disjoint at least one of the sets does not contain $x:=W(\tau)$. In the remainder of the proof we will first prove the bound \eqref{eq8} on the event that $x\not\in\wt D$, and then will use symmetry in $\wt D$ and $\wt D'$ to argue that the bound also holds on the event $x\not\in\wt D'$. It is easier to bound the sum in \eqref{eq8} on the event that $x\not\in\wt D$, since this event implies that $\tau$ does not occur during particular upcrossings of $W$ inside $\wt D$, and we can apply Lemma \ref{prop2} to describe the law of the upcrossings. 

For any path $\nu$ in the annulus $\wt D$ connecting the boundaries $\partial B(\wt\eta_1)$ and $\partial B(\wt\eta_2)$, one of the following holds: $\nu$ lies entirely within one of the four half-planes bounded the $x$- or $y$-axis,
or $\nu$ intersects both axes. By this observation we may define sets $\CG_i$ and $\CH_i$ for $1 \le i \le 4$ with union $\mcl G^*$ as follows. Let $G \in \CG_i$ if $G \in \CG^*$ and if $\nu_G$ is contained inside the half-plane $\{ (i-1) \pi / 2 \le \arg(z) \le (i+1) \pi / 2 \} $. We say $G \in \CH_i$ if $G \in \CG^*$ and $\nu_G$ connects the segments 
$ \{ r\exp((i-1) \pi / 2)\,:\, \wt\eta_1\leq r\leq \wt\eta_2 \}$ and 
$ \{ r\exp(i \pi / 2)\,:\, \wt\eta_1\leq r\leq \wt\eta_2 \}$
inside the quadrant 
$\{ (i-1) \pi / 2 \le \arg(z) \le i \pi / 2 \} $. Since
$$ \CG^* =  (\cup_{1\leq i\leq 4} \CG_i) \, \cup \, ( \cup_{1\leq i\leq 4}  \CH_i ) $$
we can treat each of the eight sets separately. By symmetry, it is enough to look at $\CG_1$ and $\CH_1$.\\

\begin{figure}[ht!]
\begin{center}
\includegraphics[scale=1]{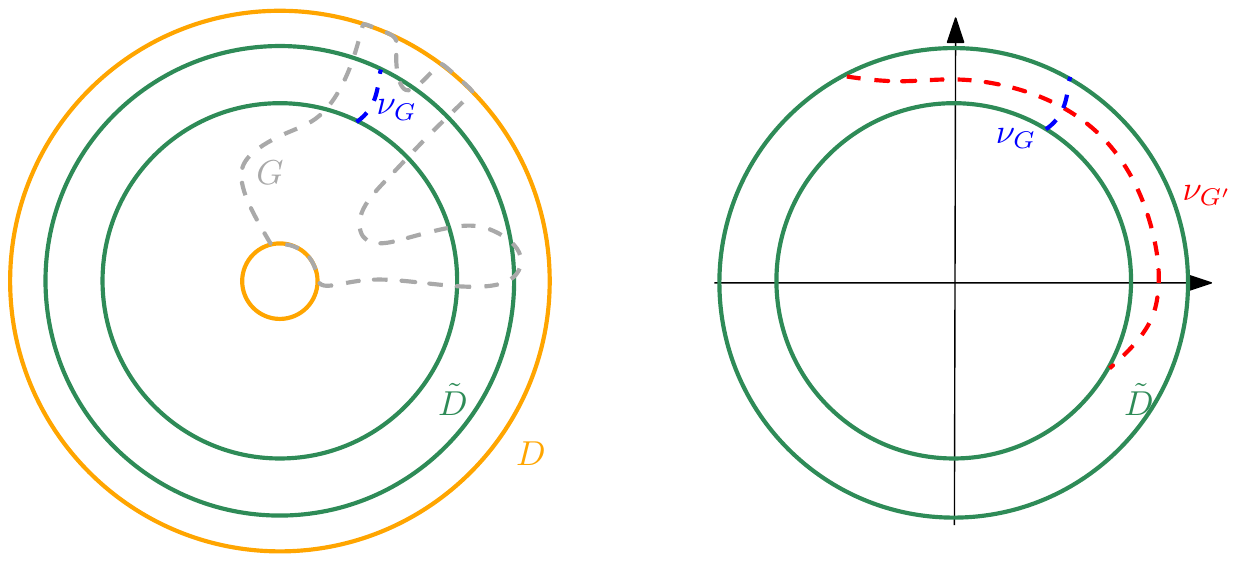}
\end{center}
\caption{Illustration of objects defined in the proof of Proposition \ref{prop3} We have $G\in\mcl G_1$ and $G'\in\mcl H_1$.
}
\label{fig-nuG2}
\end{figure}

\noindent {\bf Step 2.} Define $V=V(\wt{\eta}_1,\wt{\eta}_2)$ and $U=U(\wt{\eta}_1,\wt{\eta}_2)$ (resp.\ $V'=V(\wt{\eta}'_1,\wt{\eta}'_2)$ and $U'=U(\wt{\eta}'_1,\wt{\eta}'_2)$) by \eqref{eq27} and \eqref{eq25}, respectively. Consider the upcrossings of $W$ from $\partial(U\cup U')$ to $\partial(V\cup V')$, and condition on the initial and terminal points of the upcrossings. Furthermore, for each upcrossing we condition on whether the upcrossing was terminated before $\tau$ and whether it was started after $\tau$. Let $\wt{\E}$ be the conditional expectation. Observe that each upcrossing from $\partial(U\cup U')$ to $\partial(V\cup V')$ is either an upcrossing from $\partial U$ to $\partial V$ or an upcrossing from $\partial U'$ to $\partial V'$. Let $x_i$ and $y_i$ be the initial and terminal point, respectively, of the $i$th upcrossing $W_i$ from $\partial U$ to $\partial V$, and let $N$ be the number of such upcrossings started before time $\tau$. On the event $x\not\in \wt D$, we know that $\tau$ does not occur during the upcrossings $W_i$ for $i\leq N$. It follows by Lemma \ref{prop2} that on the event $x\not\in \wt D$, the upcrossings $W_i$ for $i\leq N$ are conditionally independent and have a Radon-Nikodym derivative of the form \eqref{eq:rn} with respect to the associated Brownian motions. By Lemma \ref{prop6} the joint conditional law of the upcrossings $W_i$ has a Radon-Nikodym derivative relative to the joint law of the associated Brownian motions which is bounded by $c_\eta^N$ for a constant $c_\eta$ satisfying $\lim_{\eta\rta 0} c_\eta=1$. By Lemma \ref{prop:UV} it follows that
\begin{equation*}
\wt{\E} \left[\sum_{G \in \CG_1} \frac{\eta^2}{\leb{G\cap U}} \one{x\not\in \wt D} \right] \preceq (c_\eta \mu)^N,
\end{equation*}
where $\mu\in(0,1)$ is a universal constant. Since $\CG_1$ consists of a single component of area at least $\mcl L(U)\asymp\eta^2$ if $N=0$ and $x\not\in \wt D$, for sufficiently small $\eta\in(0,1/2)$ such that $c_\eta\mu<1$,
\begin{equation}
\E \left[\sum_{G \in \CG_1} \frac{\eta^2}{\leb{G\cap U}} \one{x\not\in \wt D} \right] 
\preceq \sum_{k=0}^{\infty} \bP[N=k]\cdot (c_\eta\mu)^k \preceq \frac{1}{|\log\eta|},
\label{eq9}
\end{equation}
where we apply Lemma \ref{prop6} to bound $\bP[N=k]$. 

\begin{figure}[ht]
	\begin{center}
		\includegraphics[scale=1]{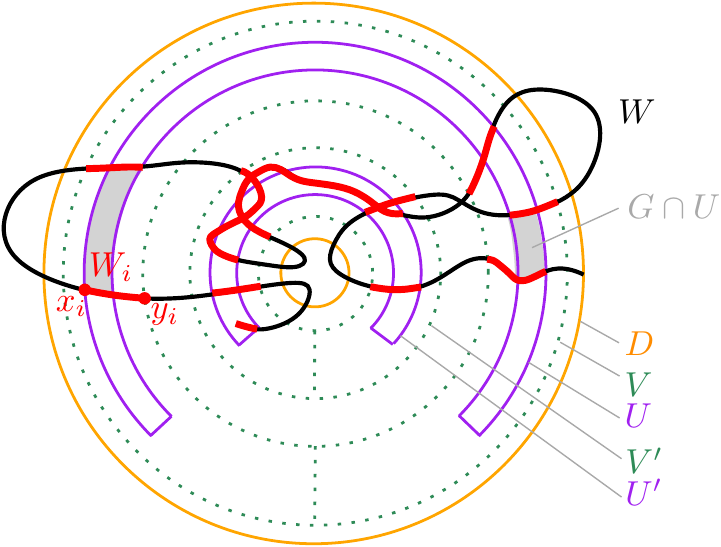}
	\end{center}
	\caption{Illustration of step 2 in the proof of Proposition \ref{prop3}. We condition on the initial and terminal points of the upcrossings shown in red. On the event that $x=W(\tau)\not\in \wt D$ (which implies $x\not\in V$) we can bound the expected inverse area of the sets $G\cap U$ for $G\in\mcl G_1$.
	}
	\label{fig-prop-ann2}
\end{figure}

Next we will apply Lemma \ref{prop7} to argue that
	\begin{equation}
	\E\left[\sum_{G \in \CG_1} \frac{\eta^2}{\leb{G\cap U}} \one{x\not\in \wt D} \right] \preceq 1.
	\label{eq36}
	\end{equation}
	Observe that
$$
\E\left[\sum_{G \in \CG_1} \frac{\eta^2}{\leb{G\cap U}} \one{x\not\in \wt D} \right] 
=\sum_{n=0}^\infty\E\left[\sum_{G \in \CG_1} \frac{\eta^2}{\leb{G\cap U}} \one{\tau\in(T_{2n},T_{2n+1})} \right],
$$
where the times $T_n$ for $n\in\N\cup\{0 \}$ are defined by \eqref{eq28}. For any fixed $n\in\N$, observe that we may define processes $\wh W_i^n$ and a set $\mcl G^*_n$ as in Lemma \ref{prop7} such that on the event $\tau\in(T_{2n},T_{2n+1})$, for any $G\in\mcl G_1$ we have $G\cap U=G^*\cap U$ for some $G^*\in\CG^*_n$. Therefore the left side of \eqref{eq36} is
$$
\le \frac{\eta^2}{\leb{U}}+ \sum_{n=1}^\infty\E\left[\sum_{G^* \in \CG_n^*} \frac{\eta^2}{\leb{G^*\cap U}} \one{\tau\in(T_{2n},T_{2n+1})} \right]
\le\frac{\eta^2}{\leb{U}}+ \sum_{n=1}^\infty\E\left[\sum_{G^* \in \CG_n^*} \frac{\eta^2}{\leb{G^*\cap U}}  \right].
$$
Lemma \ref{prop7} refers to conditioned Brownian motions, but we may average over all choices of $x_i$ and $y_i$ to get the corresponding inequality for unconditioned Brownian motions as well. Therefore the above expression is $\preceq 1$, and \eqref{eq36} follows.

By \eqref{eq9} and \eqref{eq36}, for all $\eta>0$,
\begin{equation}
\E \left[\sum_{G \in \CG_1} \frac{\eta^2}{\leb{G\cap U}} \one{x\not\in\wt D} \right] \preceq \frac{1}{|\log(\eta\wedge 1/2)|}.
\label{eq30}
\end{equation}\\

\noindent {\bf Step 3.} Next we consider $\mcl H_1$. The idea of the proof for this case is to interchange the radial and angular coordinates by the conformal map $\phi$ defined in step 1 (see Figure \ref{Dtilde3}), such that the case of $G\in\mcl H_1$ can be treated similarly as the case of $G\in\mcl G_1$. Recall the definition of $\wt V=V(\wh{\eta},\eta)$ in step 1. Then define $\wt U=U(\wh\eta,\eta)\subset\wt V$ by \eqref{eq25}.  Condition on the initial and terminal points $x_i$ and $y_i$, respectively, of the upcrossings $W_i$ from $\partial(\phi^{-1}(\wt U))$ to $\partial(\wh D\cap Q)$ (recalling that $\phi^{-1}(\wt V)=\wh D\cap Q$), on the corresponding upcrossings in $\wt D'$ instead of $\wt D$ (i.e., the upcrossings from $\partial(\phi^{-1}(\wt U'))$ to $\partial(\wh D'\cap Q)$), and on which upcrossings that are started (resp.\ completed) before time $\tau$. Let $\wt{\E}$ denote conditional expectation, and let $N$ be the total number of upcrossings $W_i$ in $\wt D$ started before time $\tau$. Since $\phi$ is conformal, the image of each $W_i$ under $\phi$ has the law of a Brownian motion started from $\phi(x_i)$, exiting $\wt V$ at $\phi(y_i)$, and killed upon exiting $\wt V$. By Lemma \ref{prop6} the joint conditional law of $W_i$ for $i\leq n$ on the event $\tau\in(T_{2n},T_{2n+1})$ has a Radon-Nikodym derivative relative to the law of the associated Brownian motions which is bounded by $c_\eta^N$, where $c_\eta\rta 1$ as $\eta\rta 0$. Since $|\phi'|\asymp 1$, we get, as in step 2 by applying Lemma \ref{prop:UV}
\begin{equation*}
\wt{\E} \left[\sum_{G \in \CH_1} \frac{\eta^2}{\leb{G\cap \phi^{-1}(\wt U)}} \one{x\not\in \wt D} \right] \preceq (c_\eta \mu)^N,
\end{equation*}
which implies by Lemma \ref{prop6} that for sufficiently small $\eta\in(0,1/2)$,
\begin{equation*}
\E \left[\sum_{G \in \CH_1} \frac{\eta^2}{\leb{G\cap \phi^{-1}(\wt U)}} \one{x\not\in \wt D} \right] 
\preceq \sum_{k=0}^{\infty} \bP[N=k]\cdot (c_\eta\mu)^k \preceq \frac{1}{|\log \eta|}.
\end{equation*}
By Lemma \ref{prop7} we get further that for all $\eta>0$
\begin{equation}
\E \left[\sum_{G \in \CH_1} \frac{\eta^2}{\leb{G\cap \phi^{-1}(\wt U)}} \one{x\not\in\wt D} \right] \preceq \frac{1}{|\log(\eta\wedge 1/2)|}.
\label{eq31}
\end{equation}

\noindent {\bf Step 4.} By \eqref{eq30}, \eqref{eq31} and similar estimates for $\mcl G_i$ and $\mcl H_i$, $i=2,3,4$,
\begin{equation*}
\E \left[\sum_{G \in \CG^*} \frac{\eta^2}{\leb{G}} \one{x\not\in\wt D} \right] \preceq \frac{1}{|\log(\eta\wedge 1/2)|}.
\end{equation*}
By a similar argument on the event $x\not\in\wt D'$, and by using that at least one of the events $x\not\in\wt D$ and $x\not\in\wt D'$ occurs, we obtain \eqref{eq8}.
\end{proof}

The next lemma will be applied to bound $R_z^2/\leb{A_z}$ when $R_z$ is smaller than $z$.
\begin{Lemma}
Let $W$ be a Brownian motion such that $W(0)=0$. For $z\in \bf C$ and $k\in\N$ define $\wt\tau_{k,z}:=\inf\{t\geq 0\,:\,|W(t)-z|\leq 2^{-k}z \}$. Then for any $\theta<1$
\begin{equation*}
\sum_{k=1}^{\infty} \int_{\bf C} \bP[\wt\tau_{k,z}<\tau]
\cdot |\log( |2^{-k}z| \wedge 1/2 )|^{-1}  
\cdot (1+|\log |2^{-k}z||)^{\theta}
\,dz<\infty.
\end{equation*}
\label{prop8}
\end{Lemma}
\begin{proof}
Define $\eta=2^{-k}|z|<|z|/2$. Then (see \eqref{eq32} for a similar estimate)
\begin{equation}
\bP[\wt\tau_{k,z}<\tau] 
= \frac{ \int_{0}^{\infty} \int_{B(\eta)} \exp(-t) t^{-1} \exp(-|z-w|^2/(2t))\,dw\,dt }
{ \int_{0}^{\infty} \int_{B(\eta)} \exp(-t) t^{-1} \exp(-|\eta-w|^2/(2t))\,dw\,dt }. 
\label{eq33}
\end{equation}
For $|z|<1$ the numerator on the right side of \eqref{eq33} is
\begin{equation*}
\preceq \eta^2\int_{0}^{|z|^{2}} t^{-1}\exp(-||z|-\eta|^2/2t) \,dt + \eta^2\int_{|z|^{2}}^{1}t^{-1}\,dt + \eta^2\int_{1}^{\infty}\exp(-t)\,dt \preceq \eta^2(1+|\log |z||).
\end{equation*}
For $\eta<1$ the denominator on the right side of \eqref{eq33} is
\begin{equation}
\succeq \eta^2+\int_{\eta^2}^{1} \eta^2 t^{-1}\,dt\asymp \eta^2(1+|\log\eta|),
\label{eq34}
\end{equation}
so for $|z|,\eta<1$ we have $\bP[\wt\tau_{k,z}<\tau]\preceq (|\log |z||+1)/(|\log\eta|+1)$. 

Letting $f(t,w)$ denote the integrand in the numerator of \eqref{eq33}, we see that for $|z|\geq 1$ and any fixed $\ep>0$ the numerator of \eqref{eq33} is 
\begin{equation*}
\begin{split}
\preceq ||z|-\eta|^{1+\ep}&\eta^2\cdot\sup_{t\geq 0, w\in B(\eta)} f(t,w)+\eta^2 \int_{||z|-\eta|^{1+\ep}}^{\infty}\exp(-t)\,dt\\
&\preceq ||z|-\eta|^{1+\ep}\eta^2\cdot||z|-\eta|^{-1+\ep} \exp(-2^{1/2}||z|-\eta|)
+ \eta^2 \exp(-||z|-\eta|^{1+\ep})\\
&\preceq \eta^2\exp(-2^{1/2-\ep}||z|-\eta|),
\end{split}
\end{equation*} 
where the implicit constant depends on $\ep$. Choosing $\ep=1/10$ it follows that for $|z|\geq 1$ and $\eta<1$ we have $\bP[\wt\tau_{k,z}<\tau]\preceq \exp(-2^{1/2-1/10}|z|) (1+|\log\eta|)^{-1}$. For $\eta\geq 1$ the denominator of \eqref{eq33} is $\succeq 1$, so for $|z|,\eta\geq 1$ we have $\bP[\wt\tau_{k,z}<\tau]\preceq \exp(-2^{1/2-1/10}||z|-\eta|)\eta^2\preceq \exp(-2^{1/2-1/10-1}|z|)|z|^2$.

The lemma follows by the above bounds for $\bP[\wt\tau_{k,z}<\tau]$,
\begin{equation*}
\begin{split}
\sum_{k=1}^{\infty} \int_{\bf C} &\bP[\wt\tau_{k,z}<\tau] 
\cdot |\log( |2^{-k}z| \wedge 1/2 )|^{-1}  
\cdot (1+|\log |2^{-k}z||)^{\theta}
\,dz
\preceq\sum_{k=1}^{\infty} \int_{\bf D} \bP[\wt\tau_{k,z}<\tau] |\log |2^{-k}z||^{-1+\theta}  \,dz\\
&\qquad + \int_{\bf C-\bf D} \sum_{k=1}^{\lfloor\log |z|/\log 2 \rfloor} \bP[\wt\tau_{k,z}<\tau] \cdot (1+|\log |2^{-k}z||)^{\theta} \,dz\\
&\qquad+ \int_{\bf C-\bf D} \sum_{k=\lfloor\log |z|/\log 2 \rfloor+1}^{\infty} \bP[\wt\tau_{k,z}<\tau] (|\log |2^{-k}z||+1)^{-1+\theta}  \,dz\\
\preceq&\,
  \int_{\bf D} \sum_{k=1}^{\infty} k^{\theta-2} (1+|\log |z||) \,dz
 +\int_{\bf C-\bf D} \exp(-|z|/10)|z|^2 (1+\log |z|)^{1+\theta} \,dz\\
 &\qquad+ \int_{\bf C-\bf D} \sum_{k=1}^{\infty} k^{\theta-2} \exp(-|z|/10) \,dz\\
<&\,\infty.
\end{split}
\end{equation*}
\end{proof}

We are now ready to prove Proposition \ref{prop1}.
\begin{proof}[Proof of Proposition \ref{prop1}]
We will bound the integrand in the statement of the proposition for each fixed $z\in\bf C$. By rotational invariance it is sufficient to consider $z>0$. Fix $z>0$, and for $k\in\bf Z$ define
\begin{equation*}
X_{k}:=\frac{R_z^2}{\mcl L(A_z)}(|\log R_z|+1)^\theta \one{2^{-(k+1)}\leq R_z/z<2^{-k}}.
\end{equation*}
Then write
\begin{equation}
\begin{split}
\E\left[\frac{R_z^2}{\mcl L(A_z)}(|\log R_z|+1)^\theta\right] =
\sum_{k\in\Z,\,k\leq -3}^\infty \E\left[X_{k}\right]
+\sum_{k=-2}^{0} \E\left[X_{k}\right]
+\sum_{k=1}^\infty \E\left[X_{k}\right].
\end{split}
\label{eq12}
\end{equation}
In order to bound the terms on the right side we first observe that for any $\eta>0$ and with $f(t)=\exp(-t)$ being the probability density function of $\tau$, scale invariance of Brownian motion gives
\begin{equation}
\begin{split}
&\bP\left[ \sup_{t\in[0,\tau]} |W(t)|>\eta \right]
= \int_{0}^{\infty} f(s) \bP\left[ \sup_{t\in[0,1]} |W(t)|>\eta s^{-1/2} \right]\,ds
\preceq \int_{0}^{\infty} \exp(-s) s^{1/2}\eta^{-1} \exp(-\eta^2/(2s))\,ds\\
&\qquad\leq \int_{0}^{\eta} \eta^{-1/2} \exp(-\eta^2/(2s))\,ds
+ \int_{\eta}^{\infty} s^{1/2}\eta^{-1} \exp(-s) s\,ds
\preceq \exp(-\eta/3).
\end{split}
\label{eq15}
\end{equation}

We now bound the terms in the first sum on the right side of \eqref{eq12}. See Figure \ref{fig-ecc-int1}. Throughout this paragraph $k\leq -3$. We decompose $W$ into two parts: $W_1=W|_{[0,\tau']}$ and $W_2=W|_{[\tau',\infty)}$, where $\tau'=\inf\{t\geq 0\,:\,|W(t)-z|\geq 2^{-(k+2)}z \}$. Conditioned on $W(\tau')$ the two curves $W_1$ and $W_2$ are independent. On the event $2^{-(k+1)}\leq R_z/z<2^{-k}$ we know that $\tau'<\tau$. Conditioned on $\tau>\tau'$ the random variable $\tau-\tau'$ has the law of a unit rate exponential random variable, so by Proposition \ref{prop3} with $\eta=2^{-(k+1)}z$ and $\mcl G^*$ defined as in the statement of the proposition (with the annulus centered at $z$), the estimate \eqref{eq8} holds. On the event $2^{-(k+1)}\leq R_z/z<2^{-k}$ there is a $G\in\mcl G^*$ such that $G\subset A_z$, so $\E[X_k\,|\,\tau>\tau']\preceq (1+|\log( 2^{-k}z)|)^\theta$. We conclude by \eqref{eq15} that
\begin{equation*}
\E\left[X_{k}\right]
\preceq \bP[ \tau>\tau' ] \cdot \E\left[X_k\,|\,\tau>\tau'\right]
\preceq \exp(-2^{-k-3}z/3)(1+|\log(2^{-k}z)|)^\theta,
\end{equation*}
so the first sum on the right side of \eqref{eq12} is finite and integrable.

\begin{figure}[ht]
\begin{center}
\includegraphics[scale=1]{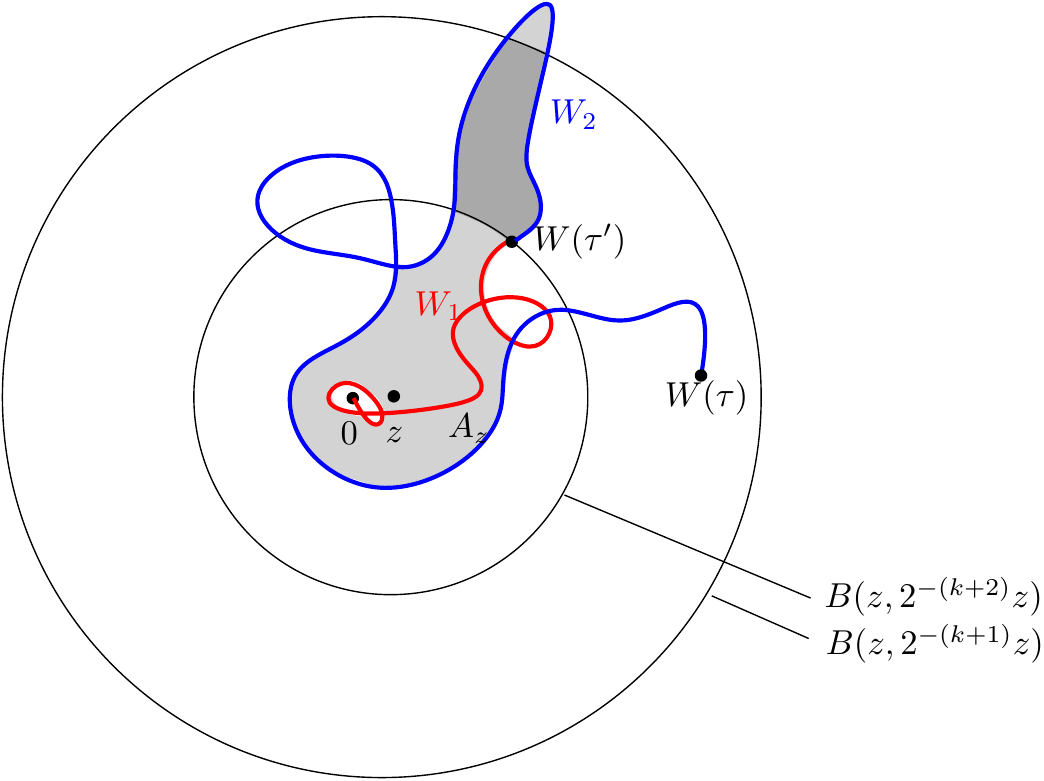}
\end{center}
\caption{We bound the eccentricity of $A_z$ on the event that $2^{-(k+1)}\leq R_z/z<2^{-k}$ for $k\leq -3$.
}
\label{fig-ecc-int1}
\end{figure}

Next we bound the terms in the second sum on the right side of \eqref{eq12}. Throughout the paragraph $k\in\{-2,-1,0\}$. We decompose $W$ into three parts: $\wt W_1=W|_{[0,\wt \tau]}$, $W_2=W|_{[\wt\tau,\wh\tau]}$ and $W_3=W|_{[\wh\tau,\infty)}$, where 
$\wt\tau=\inf\{t\geq 0\,:\,|W(t)|\geq 2^{-1}z \}$ and 
$\wh\tau=\inf\{t\geq \wt{\tau}\,:\,|W(t)-z|\leq 2^{-5}z \}$. Conditioned on $W(\wt{\tau})$ and $W(\wh{\tau})$ the processes $W_1$, $W_2$ and $W_3$ are independent. On the event $2^{-(k+1)}\leq R_z/z<2^{-k}$ we have $\wt{\tau}<\tau$. Therefore
\begin{equation}
\begin{split}
\E\left[X_k\right]
&= \bP[\wt{\tau}<\tau]\cdot \E\left[X_k\,|\,\wt{\tau}<\tau \right]\\
&=\bP[\wt{\tau}<\tau]\cdot
\big(\bP[\wh{\tau}<\tau\,|\,\wt{\tau}<\tau]
\cdot \E\left[X_k\,|\,\wt{\tau},\wh{\tau}<\tau \right]
+
\bP[\wh{\tau}\geq \tau\,|\,\wt{\tau}<\tau]
\cdot \E\left[X_k\,|\,\wt{\tau}<\tau,\wh{\tau}\geq\tau \right]\big).
\end{split}
\label{eq16}
\end{equation}
By \eqref{eq15} we have $\bP[\wt{\tau}<\tau]\preceq \exp(-z2^{-1}/3)$. By Proposition \ref{prop3} and since $\tau-\wh{\tau}$ has the law of a unit rate exponential random variable conditioned on $\wh{\tau}<\tau$, we have $\E[X_k\,|\,\wt\tau,\wh{\tau}<\tau]\preceq |\log(z\wedge 1/2)|^{-1}\cdot (1+|\log z|)^\theta$, since with $\eta=2^{-5}z$ there is a $G\in\mcl G^*$ (with $\mcl G^*$ as in Proposition \ref{prop3}) such that $G\subset A_z$. On the event $\wh{\tau}\geq\tau$ we have $X_k\preceq (1+|\log z|)^\theta$, since $B(z,2^{-5}z)\subset A_z$. Inserting these estimates into \eqref{eq16} we get
\begin{equation*}
\begin{split}
\E\left[X_k\right]
&\preceq \exp(-z/6)
\cdot(1+|\log z|)^\theta,
\end{split}
\end{equation*}
so the second sum  on the right side of \eqref{eq12} is finite and integrable.

We now bound the terms in the third sum on the right side of \eqref{eq12}. See Figure \ref{fig-ecc-int3}. We proceed exactly as when bounding the terms in the second sum, except that we define $\wt\tau=\inf\{t\geq 0\,:\,|W(t)-z|\leq 2^{-k}z \}$ and 
$\wh\tau=\inf\{t\geq \wt{\tau}\,:\,|W(t)-z|\leq 2^{-(k+1)}z \}$. As before we have $\wt{\tau}<\tau$ on the event $2^{-(k+1)}\leq R_z/z<2^{-k}$, so \eqref{eq16} still holds. By Proposition \ref{prop3} applied with $W_3$ and $\eta=2^{-(k+1)}z$, we have $\E[X_k\,|\,\wt\tau,\wh{\tau}<\tau]\preceq |\log\big((2^{-k}z)\wedge (1/2)\big)|^{-1}\cdot (1+|\log 2^{-k}z|)^\theta$, since on the event $2^{-(k+1)}\leq R_z/z<2^{-k}$ there is a $G\in\mcl G^*$ (with $\mcl G^*$ as in Proposition \ref{prop3}) such that $G\subset A_z$. By Lemma \ref{prop6} applied with $U=B(z,2^{-(k+1)}z)$ and $V=B(z,2^{-k}z)$ we have $\bP[\wh{\tau}\geq\tau\,|\,\wt{\tau}<\tau]\preceq |\log\big((2^{-k}z)\wedge (1/2)\big)|^{-1}$. On the event $\wh{\tau}\geq\tau$ we have $X_k\preceq (1+|\log(2^{-k}z)|)^\theta$, since $B(z,2^{-(k+1)}z)\subset A_z$. Inserting these estimates into \eqref{eq16} we get
\begin{equation*}
\begin{split}
\E\left[X_k\right]
= \bP[\wt\tau<\tau]\cdot |\log\big((2^{-k}z)\wedge (1/2)\big)|^{-1}
\cdot (1+|\log (2^{-k}z)|)^\theta.
\end{split}
\end{equation*}
It now follows from Lemma \ref{prop8} that the third sum on the right side of \eqref{eq12} is finite and integrable.
\end{proof}

\begin{figure}[ht]
\begin{center}
\includegraphics[scale=1]{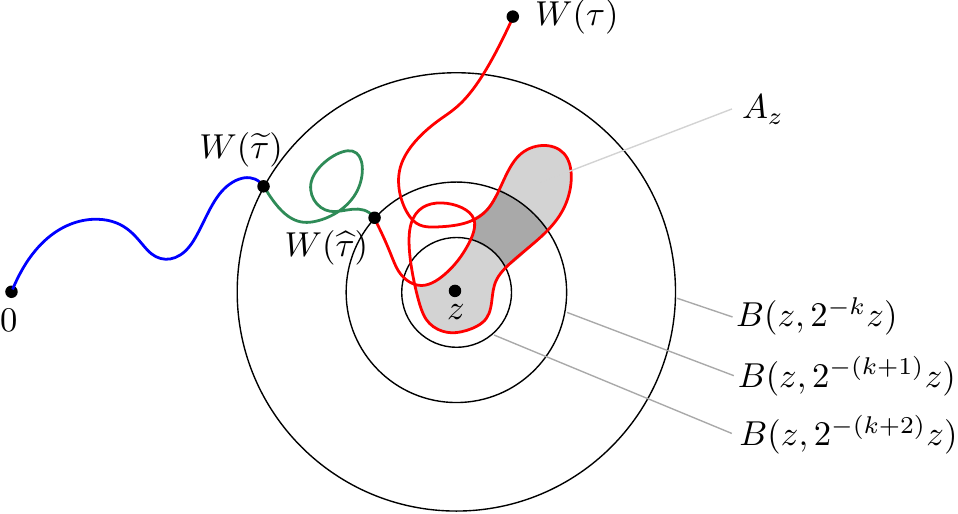}
\end{center}
\caption{We bound the eccentricity of $A_z$ on the event that $2^{-(k+1)}\leq R_z/x<2^{-k}$ for $k\geq 1$.
}
\label{fig-ecc-int3}
\end{figure}

\section{Estimates on the in-radius of Brownian components}
\label{sec:prop5}
In this section we will first prove Proposition \ref{prop5}, and then we will give a short independent proof (using a result by \cite{legall}) that the analogue of Theorem \ref{thm2} with $r(i)$ instead of $R(i)$ holds. We will need the following lemma in our proof of Proposition \ref{prop5}, which is an immediate consequence of a result by \citet*{legall}.
\begin{Lemma}
Let $A(i)$ be the $i$th largest bounded component of ${\bf C}-W([0,1])$, and let $\leb{A(i)}$ be the Lebesgue measure of $A(i)$. Then
$$\text{a.s.-}\lim_{i\rta\infty} i(\log i)^2 \leb{A(i)}= 2\pi.$$
\label{prop4}
\end{Lemma} 
\begin{proof}
For any $\ep>0$ let $\mcl N(\ep)$ be the number of components with area at least $\ep$, and observe that $\mcl N(\leb{A(i)})=i$ for any $i\in\N$. \citet*[Theorem 11.6]{legall} proved that $\mcl N(\ep)=2\pi a(\ep)\ep^{-1}|\log\ep|^{-2}$ for random variables $a(\ep)$ satisfying a.s.-$\lim_{\ep\rta 0}a(\ep)= 1$. Therefore, for any $c>0$,
\begin{equation}
\mcl N\left(\frac{c}{i(\log i)^2} \right) \cdot \frac{c}{i} = 
2\pi a\left(\frac{c}{i(\log i)^2}\right) \frac{i(\log i)^2}{c(\log i+\log(\log i)^2 - \log c)^2} \cdot \frac{c}{i} \rta
2\pi\qquad\text{a.s.\,\,as\,\,}i\rta\infty.
\end{equation}
If $c>2\pi$ we see that $\mcl N\left(c (i(\log i)^2)^{-1} \right)<i$ for all sufficiently large $i$; similarly $c<2\pi$ implies that $\mcl N\left(c (i(\log i)^2)^{-1} \right)>i$ for all sufficiently large $i$. These observations imply the lemma.
\end{proof}

\begin{proof}[Proof of Proposition \ref{prop5}]
We assume the components of ${\bf C}-W([0,1])$ are ordered such that $(\leb{A(i)})_{i\in\N}$ is decreasing. \citet*[Theorem 14]{wwthesis} proved that there is a constant $c>0$ such that 
\begin{equation*}
\text{a.s.-}\lim_{n\rta\infty}
\frac{1}{n}\left|\left\{ i\in\{1,\dots,n\}\,:\,10 r(i)\geq \leb{A(i)}^{1/2} \right\}\right| \rta 3c,
\end{equation*}
since under the law $\mcl L_1$ (which describes the limiting law of the components $A(i)$), ten times the in-radius is larger than the square-root of the area with positive probability. Therefore, a.s.\ for all sufficiently large $n\in\N$,
\begin{equation}
\left|\left\{ i\in\{n,\dots,2n-1\}\,:\,10r(i)\geq \leb{A(i)}^{1/2} \right\}\right| \geq cn.
\end{equation} 
By Lemma \ref{prop4}, $\leb{A(n)}n(\log n)^2>1$ a.s.\ for all sufficiently large $n$. Therefore, a.s.\ for all sufficiently large $n$, 
\begin{equation}
\begin{split}
\sum_{n\leq i< 2n} r(i)^2|\log r(i)| &\succeq \left|\left\{ i\in\{n,\dots,2n-1\}\,:\,10 r(i)\geq \leb{A(i)}^{1/2} \right\}\right|\cdot \leb{A(n)} |\log \leb{A(n)}|\\
&\geq cn \cdot \frac{\log n + 2\log\log n}{n (\log n)^2} \\
&\asymp (\log n)^{-1}.
\end{split}
\end{equation}
Since $\sum_{k=1}^\infty |\log 2^k|^{-1}=\infty$ the result (i) follows as
\begin{equation}
\sum_{i=1}^\infty r(i)^2|\log r(i)|
= \sum_{k=0}^\infty \sum_{2^k\leq i< 2^{k+1}} r(i)^2|\log r(i)|
=\infty.
\end{equation}

\end{proof}

We will end this section with a short independent proof that, with the notation of Proposition \ref{prop5}, the following holds for any $\theta<1$
\begin{equation}
\E\left[\sum_{i=1}^\infty r(i)^2|\log r(i)|^\theta\right]<\infty.
\label{eq40}
\end{equation}
Observe that this result is immediate by Theorem \ref{thm2}, since $r(i)\leq R(i)$ for all $i\in\N$, and since Theorem \ref{thm2} with $\theta=0$ implies that the number of components $A(i)$ with an out-radius above (say) $1/10$ has a finite expectation. We will give an alternative proof of \eqref{eq40} by using a result of \cite{legall}. Since $\pi r(i)^2\leq \leb{A(i)}$ for all $i\in\N$, it is sufficient to prove that for any $\theta<1$,
\begin{equation}
\E\left[\sum_{i=1}^{\infty} \leb{A(i)}(|\log \leb{A(i)}|+1)^\theta \right] <\infty.
\label{eq35}
\end{equation}
For any $k\in\Z$ define
\begin{equation*}
\I_k :=\{ i\in\N\,:\,\leb{A(i)}\in[2^k,2^{k+1}) \},\qquad
U_k:= \sum_{i\in\I_k} \leb{A(i)}.
\end{equation*}
First we bound the sum in \eqref{eq35} only considering $i\in \I_+$ for $\I_+:=\bigcup_{k\geq 0}\I_k$. Define $X:=\sup_{t\in[0,1]}|W(t)|$, and note that $\bP[X>M]\preceq \exp(-M^2/2)$ by the Gaussian tail bound. Since $X<1/2$ implies $\mcl I_+=\emptyset$, and since $X<2^n$ implies $|\I_+|\preceq 2^{2n}$ and $\leb{A(i)}\preceq 2^{2n}$ for all $i\in\I_+$, we have
\begin{equation*}
\begin{split}
\E\left[\sum_{i\in\I_+} \leb{A(i)}(|\log \leb{A(i)}|+1)^\theta \right]
&\leq \sum_{n=0}^\infty \E\left[\one{2^{n-1}\leq X< 2^n}
\sum_{i\in\I_+} \leb{A(i)}(|\log \leb{A(i)}|+1)^\theta \right]\\
&\preceq \sum_{n=0}^\infty \exp(-2^{2n-3})\cdot 2^{2n}\cdot 2^{2n}(|\log 2^{2n}|+1)^\theta<\infty.
\end{split}
\end{equation*}
Next we consider the sum in (ii) for $i\in\mcl I_-:=\bigcup_{k< 0}\I_k$. \citet*[Proposition 3]{legall} proved that $\E[U_{-k}]\asymp k^{-2}$ for all $k\in\N$, so
\begin{equation*}
\begin{split}
\E\left[\sum_{i\in\I_-} \leb{A(i)}(|\log \leb{A(i)}|+1)^\theta \right]
&\preceq \sum_{k=1}^{\infty} \E\left[|\I_{-k}|\cdot 2^{-k}|\log 2^{-k}|^\theta \right]\\
&\asymp \sum_{k=1}^{\infty} \E\left[2^{k} U_{-k}\cdot 2^{-k}|\log 2^{-k}|^\theta \right]
\asymp \sum_{k=1}^{\infty} k^{-2+\theta} 
<\infty.
\end{split}
\end{equation*}
Combining the bounds for the sum over $\mcl I_+$ and $\I_-$, respectively, completes the proof of \eqref{eq40}.

\appendix
\section{Basic estimates for planar Brownian motion}

In the first part of the appendix we state some standard results on hitting probabilities for planar Brownian motion. In the second part of the appendix we will consider a conditioned Brownian motion $W$ in a domain $D$ and a collection $\mcl Q$ of polar rectangles touching $\partial D$, and we will prove (under certain assumptions on $W$, $D$ and $\mcl Q$) uniform exponential bounds for the number of polar rectangles hit by $W$ before hitting $\partial D$.

\subsection*{Standard estimates for planar Brownian motion}
Recall throughout the section that if $K\subset\bf C$ and $W$ is a Brownian motion, then $T_K=\inf\{t\geq 0\,:\,W(t)\in K \}$ denotes the time $W$ hits $K$.

\begin{Theorem}
{\bf (Beurling's theorem).} Consider a set $K \subset {\bf C}$,
its radial projection $\Pi (K) = \{ |z|: z \in K \}$
and the ball $B(R)$. Then a Brownian motion started at
0 and killed upon exiting the ball is more likely to hit $K$ than $\Pi(K)$, i.e., 
$$ \bP_0[T_K \le T_{\partial B(R)}] 
\ge \bP_0[T_{\Pi(K)} \le T_{\partial B(R)}]. $$
\end{Theorem}
\noi For a proof see for example \cite{wwbeurling}. \\

We will also need some standard bounds for Brownian motion. Intuitively, they estimate
the probability that a Brownian motion started near a boundary will ``escape'' and go for
a certain distance without hitting the boundary.

\begin{Proposition}\label{estunitdisk}
Consider a Brownian motion started at $z$ inside the disk $B(R)$.
Then the probability that a Brownian motion will travel a distance at least $r$
from $z$ before hitting the disk boundary satisfies:
\begin{equation*}
\pz ( T_{\partial B(z,r)} < T_{\partial B(R)} ) \le 2 \delta / r,
\end{equation*}
where  $\delta = R - |z|$ the distance from $z$ to the boundary. 
\end{Proposition}

\begin{proof} 
This is a standard result, but for completeness we sketch the proof. Let $L$ be the tangent to $\partial B(R)$ that is closest to $z$, and assume without loss of generality that $L$ is the $x$-axis and $z=\delta i$. It is enough to estimate the probability that a Brownian motion travels distance $r$ before hitting $L$. We can use the M\"obius map $\phi(z)=(z+d)/(z-d)$, $d=(r^2-\delta^2)^{1/2}$, to send $\partial B(z,r)$ and $L$ into two axes meeting at the origin at angle $\pi/2+\alpha$, where $\alpha$ satisfies $\sin\alpha=\delta/r$. We have $\arg\phi(z)=2\alpha$, so the probability of interest is equal to exactly $2\alpha/(\pi/2+\alpha)$. The proposition follows by using $\sin \alpha \ge (2/\pi) \alpha$.
\end{proof}

\begin{Proposition}\label{estBMsqrt}
If $0<a<b\le1$, then the probability that a Brownian motion started at 0 hits the unit disk
before hitting the segment $[a,b]$ satisfies
\begin{equation*}
\pz ( T_{[a,b]} > T_{\partial {\bf D}} ) \le 2 \sqrt{a/b}.
\end{equation*}
\end{Proposition}

\begin{proof}
Again we only sketch the proof since this is a standard result.
By killing the Brownian motion on $\partial B(b)$ and using scale invariance we can assume $b=1$.
Using the M\"obius map $z\mapsto(z-a)/(1-az)$ we can reduce to the problem of estimating the probability that a
Brownian motion started at $-a$ exits the disk before hitting $[0,1]$. The conformal map $x \mapsto \sqrt x$
sends ${\bf D} - [0,1]$ into a half-disk and $-a$ into $i\sqrt a$.  Now we use the same arguments
as in the proof of Proposition~\ref{estunitdisk}.
\end{proof}

Finally, we state an estimate for the probability that a Brownian motion started ``near the middle'' of the unit disk hits a small ball that touches the boundary.

\begin{Proposition}\label{from_mid_to_edge}
Let $x \in (0,1)$, $z \in {\bf D}$ and $\eps \ge 1-x$. Then
\begin{equation*}
\pz ( T_{B(x, \eps)} < T_{\partial {\bf D}} ) \le 4 \eps / |z-x|
\end{equation*}
\end{Proposition}

\begin{proof}
Assume first that $\eps = 1-x$ so $B(x,\eps)$ is tangent to ${\bf D}$
and write $1-z = r \exp(i\theta)$.
Then the map $\phi(y) = 1/(1-y)$ sends both circles into vertical lines and by conformal
invariance it follows that
$$ \pz ( T_{B(x, \eps)} < T_{\partial {\bf D}} ) = (2\cos\theta / r - 1) / (1/\eps - 1). $$
If $\eps \le 1/2$ then the last expression is bounded above by $4 \eps / r$; if
$\eps > 1/2$ then the same bound holds automatically, since $r \le 2$ always.
We have $|z-x|<|z-1|=r$ unless $z$ is within $\eps$ of the unit circle.
If $|z-x|<4\eps$ there is nothing to prove. And otherwise the Brownian motion must travel distance
at least $|z-x|/2$ from $z$ in order to hit $B(x,\eps)$,
and we can apply Proposition~\ref{estunitdisk}.

The case when $\eps > 1-x$ can be reduced to the previous one,
by killing the Brownian motion when exiting $B(x+\eps)$ instead of the unit disk; 
the probability of hitting  $B(x,\eps)$ can only increase, and the circles are
now tangent; rescaling completes the proof.
\end{proof}

\subsection*{Estimates for the number of polar rectangles hit by a Brownian motion}

The main result of this section is Proposition \ref{exp_tail_for_la_domain}, which implies Corollary \ref{cor_tail_est}.

Consider the following setup: a bounded planar domain $D$ and a Brownian motion $W$ started at some point $x \in D$, conditioned to exit $D$ at some point $y \in \partial D$ and killed upon exiting. Inside $D$ we have a finite collection $\mcl Q$ of disjoint sets $Q_1, \ldots, Q_n$, all touching $\partial D$, having regular shapes, and not too close to each other or $x$ and $y$. Let $N$ be the number of sets that the Brownian motion visits before it is killed. We will prove that $N$ has exponential tails under certain assumptions on $D$, $x$, $y$, and $\mcl Q$. Lemma \ref{tiles} is a result for general domains, and the estimate is expressed in terms of a Riemann map $\phi:D\to\bf D$. Proposition \ref{exp_tail_for_la_domain} considers the special case where $D$ is a $\lambda$-domain (Definition \ref{def:lambdadomain}). The idea of the proof is to show that an \emph{unconditioned} Brownian motion $W$ has a uniformly positive probability of exiting $D$ every time it hits a set $Q_j$ (Lemma \ref{tiletile}), and deduce the wanted exponential decay for the conditioned Brownian motion $W$ by using that $\phi(x)$ and $\phi(y)$ are bounded away from the image of $\mcl Q$ under $\phi$.

Both in Lemma \ref{tiles} and Proposition \ref{exp_tail_for_la_domain} the set $\mcl Q$ will be a collection of polar rectangles satisfying the following property.

\begin{Definition}\label{def_nice_q}
A finite collection ${\cal Q} = \{Q_j\}$ of polar rectangles is
called {\bf nice} if:
\begin{itemize}
\item[(i)] The rectangles $Q_j$ are not too close or too far from the origin:
there exist constants $0 < \rho_1 < \rho_2$ such that $\rho_1 \le |z| \le \rho_2$
for all $j$ and all $z \in Q_j$.
\item[(ii)] The radial and angular size of each rectangle are of the same
order: there exist constants $C_1, C_2$ such that for all $j$,
\begin{equation}\label{drdtheta}
C_1 \cdot dr(Q_j) \le d\theta(Q_j) \le C_2 \cdot dr(Q_j).
\end{equation}
\item[(iii)] For all $j$, $d\theta(Q_j) \le \pi / 2$.
\item[(iv)] The rectangles $Q_j$ are disjoint and far from each other,
relative to their size: for all $i \ne j$,
$d(Q_i, Q_j) \ge dr(Q_i)$.
\end{itemize}
\end{Definition}

\begin{figure}[ht]
\begin{center}
\includegraphics[scale=1]{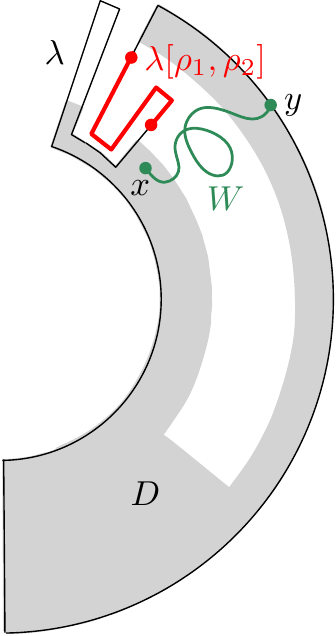}\qquad\qquad\qquad\qquad
\includegraphics[scale=1]{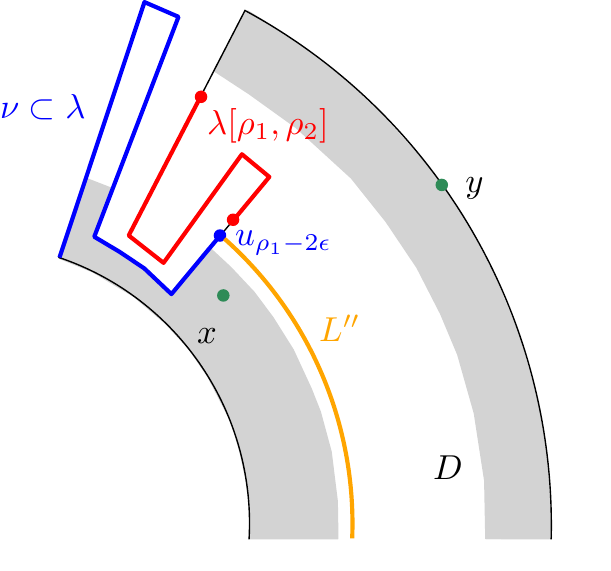}
\end{center}
\caption{Left: Illustration of the statement of Proposition \ref{exp_tail_for_la_domain}. A Brownian motion $W$ started from $x$ (contained in the gray domain) is conditioned to leave the $\lambda$-domain $D$ at $y$. A nice collection of polar rectangles $\mcl Q$ (not shown on the figure) touch the segment $\lambda[\rho_1,\rho_2]$ of $\lambda$ shown in red. Right: Illustration of step 5 in the proof of Proposition \ref{exp_tail_for_la_domain}.
}
\label{fig-prop-app}
\end{figure}

Before we can state the proposition we will introduce some notation. Let $\la$ be a curve as in the definition of a $\la$-domain with parameters $\eta_1,\eta_2$. For any $\rho_1, \rho_2  \in [\eta_1, \eta_2]$ satisfying $\rho_1<\rho_2$, we will let $\la[\rho_1,\rho_2]$ denote a connected segment of $\la$ defined as follows. For any $\rho \in [\eta_1, \eta_2]$, let
$$ \theta(\la, \rho) = \inf \{ \theta > 0 : \rho \exp(i \theta) \in \la \}, $$ 
and let $u_\rho = \rho \exp (i \theta(\la, \rho))$.
Then define $\la[\rho_1,\rho_2]$ to be the connected segment of $\la$
between $u_{\rho_1}$ and $u_{\rho_2}$.\\ 

\begin{Proposition}\label{exp_tail_for_la_domain}
Consider the following objects:
\begin{itemize}
\item $D$ is a $\la$-domain with parameters $\eta_1, \eta_2$,
\item ${\cal Q} = \{Q_j\}$ is a nice collection of polar rectangles, with parameters
$\rho_1 > \eta_1, \rho_2 < \eta_2, C_1, C_2$, such that each $Q_j$ 
touches $\la[\rho_1, \rho_2]$  along one of its line sides,
\item $x \in D$ is a point such that either $\arg(x) \le -\pi / 4$, or there exists a $\delta > 0$ such that $x$ can be connected to the positive real line by a curve contained in $\{ z\in D\,:\, |z|\not\in (\rho_1 - \delta,\rho_2 + \delta) \}$,
\item $y$ is any point in $\partial D - \la$.
\end{itemize}
Let $W$ be a Brownian motion started at $x$, conditioned to exit $D$ at $y$ and
killed upon exiting. Let $N$ be the number of rectangles in $\cal Q$ hit by $W$. Then $N$ has
exponential tails:
\begin{equation*}
\pxy(N \ge t) \le C \mu^t \quad \forall t \ge 1,
\end{equation*}
where $C>0$ and $\mu\in(0,1)$ are constants depending only on 
$\eta_1, \eta_2, \rho_1, \rho_2, C_1, C_2$ and $\delta$
(but do not depend on the shape of $\la$ or on the number and sizes of the
rectangles in $\cal Q$).
\end{Proposition}

\begin{figure}[ht]
\begin{center}
\includegraphics[scale=1]{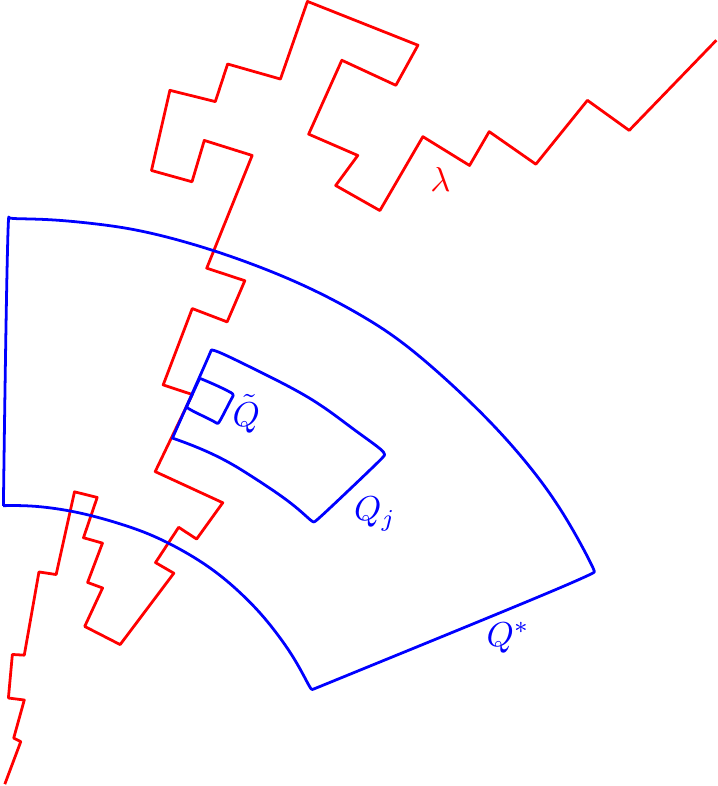}
\end{center}
\caption{
The three polar rectangles in the proof of Lemma~\ref{tiletile}:
$Q_j$, $\tQ$ and $Q^*$. The rectangle $Q_j$ touches $\la$ along its left side.
}
\label{qqq}
\end{figure}

\noi For a nice collection of polar rectangles that touch a curve
(e.g. the boundary of a domain), we can show that a Brownian motion started inside one
of the rectangles is likely to hit the curve before hitting
another rectangle.
\begin{Lemma}\label{tiletile}
Let $\la$ be a connected curve and let ${\cal Q} = \{Q_j\}$ be 
a nice collection of polar rectangles, so that each $Q_j$ touches $\la$
along one of its line sides.
Assume $\la$ is large with respect to the rectangles, so its diameter
$ |\la| \ge 2 \, dr(Q_j) $ for all $j$.
Let $Q = \cup Q_j$. Then for all $j$ and all $z \in Q_j$,
\begin{equation*}
\pz ( T_\la < T_{Q - Q_j} ) \ge c
\end{equation*}
for some constant $c>0$ that depends only on $\rho_1, \rho_2, C_1$ and $C_2$.
\end{Lemma}

\begin{proof}
The proof involves several steps. For each rectangle $Q_j$ we construct
a smaller rectangle $\tQ$ and a larger one $Q^*$. We show that a Brownian motion started
inside $Q_j$ is likely to hit the small rectangle, and once that happens it is likely to get killed before exiting the large rectangle. \\

\noi {\bf Step 1.} Let $Q_j = q(r, \theta, dr, d\theta)$.
One of the two line sides of $Q_j$ must touch $\la$;
assume without loss of generality that this is the side
$L = \{ r'\exp(i(\theta+d\theta)) \, | \, r \le r' < r + dr\}$ (Figure~\ref{qqq}).
Let $\eps>0$ satisfy $\eps<(12 \rho_1 C_2)^{-1}\wedge (1/2)$ and consider the sub-rectangle
$$ \tQ = \{ r'\exp(i\theta') \, | 
\, r + \frac 13 \cdot dr \le r' < r + \frac 23 \cdot dr, \, 
\theta + (1-\eps) \cdot d\theta \le \theta' < \theta + d\theta \}. $$
Then $\tQ$ touches the edge $L$ and is shrunk by factors of
$1/3$ and $\eps$ with respect to $Q_j$. Since $L$
touches $\la$, all points in $\tQ$ are close to $\la$.
Indeed, let $z \in \tQ$ and let $z_L$ the unique point on $L$
with $|z| = |z_L|$. Then $d(z, \la) \le d(z_L, \la) + d(z, z_L)$, so
\begin{equation}\label{eqtt1}
d(z, \la) \le \frac 23 \cdot dr + \eps \cdot \rho_2 \cdot d\theta
\le \frac 34 \cdot dr
\quad \forall z \in \tQ.
\end{equation}
Since $\cal Q$ is nice, we also have
\begin{equation}\label{eqtt2}
d(z, Q-Q_j) \ge dr \quad \forall z \in \tQ.
\end{equation}

Now we use Beurling's theorem for Brownian motion started at $z$ and killed upon
exiting $B(z, dr)$. The radial projection of $\la$ (with center $z$) is
some interval $\Pi(\la)$. From \eqref{eqtt1}, it must contain some
number below $(3/4) \cdot dr$; since $|\la|>2 \, dr$, it must also contain some number
above $dr$. Hence $(3/4 \cdot dr, dr)  \subset \Pi(\la)$ so Beurling's theorem guarantees
that
\begin{equation}\label{eqc1}
\pz ( T_\la < T_{Q-Q_j} ) \ge c_1 > 0 \quad\forall z \in \tQ
\end{equation}
where $c_1$ is, by scale invariance, the probability that a Brownian motion started at the origin hits
the segment $(3/4,1)$ on the $x$-axis before hitting the unit circle. \\

\noi {\bf Step 2.}
Now fix $\delta \in (0,1)$ and construct a polar rectangle $Q^*$ by padding
$Q_j$ on all sides with bands of relative size $\delta$ (Figure~\ref{qqq}):
\begin{equation*}
Q^* = \{ r'\exp(i\theta') \, | 
\, r- \delta \cdot dr \le r' < r + (1+\delta) \cdot dr, \, 
\theta - \delta \cdot d\theta \le \theta' <\theta + (1+\delta) \cdot d\theta \}.
\end{equation*}
Clearly $Q^*$ contains $Q_j$. We claim that if $\delta$ is small enough, then $Q^*$ will not intersect any of the other polar rectangles. Since the rectangles are at least $dr$ apart, it will be enough to show that any point $u \in Q^*$ is within $dr$ of $Q_j$. Indeed, let $v$ be the point of $\partial Q_j$ that is nearest $u$. The points $u$ and $v$ can be connected by a path made of an arc of length at most $\delta \cdot (\rho_2+\delta dr) \cdot d\theta\leq 2\delta\rho_2 d\theta$ (where the inequality holds since $\delta<1<\rho_2/dr$), followed by a straight line of length at most $\delta \cdot dr$. Hence it is enough to choose $\delta = 1 / (1+2C_2 \rho_2)$. \\

\noi {\bf Step 3.}
Hence a Brownian motion started at some $z \in Q_j$ must exit $Q^*$ before hitting
$Q-Q_j$. We claim a Brownian motion started at $z$ is likely to hit $\tQ$ before exiting
$Q^*$:
\begin{equation}\label{eqc2}
\pz ( T_{\tQ} < T_{\partial Q^*} ) \ge c_2 > 0 \quad\forall z \in Q_j.
\end{equation}
This is clear intuitively; the proof involves a series of simple but tedious
calculations. We use the polar representation of Brownian motion; 
we need to estimate the probability 
that a Brownian motion started at $(\log |z|, \arg z)$ hits the (cartesian) rectangle
$$[\log(r+(1/3)dr), \log(r+(2/3)dr)] \times 
[\theta + (1-\eps) \cdot d\theta, \theta + d\theta]$$
before exiting the rectangle
$$[\log(r-\delta \cdot dr), \log(r+(1+\delta) \cdot dr)] \times 
[\theta - \delta \cdot d\theta, \theta + (1+\delta) \cdot d\theta].$$
Now we use translation invariance to subtract $\log r$ and $\theta$,
and scale invariance to factor out $d\theta$. The two
rectangles transform to
$$[\log(1+(1/3)dr/r) / d\theta, \log(1+(2/3)dr/r) / d\theta] \times [1-\eps,1]$$
and
$$[\log(1-\delta \cdot dr/r) / d\theta, \log(1+(1+\delta) \cdot dr/r) / d\theta] \times 
[-\delta, 1+\delta] $$
and the starting point $z$ transforms to a point 
$z' \in [0, \log(1+dr/r) / d\theta] \times [0,1]$.
From the inequality
$$ t / (x+t) \le \log(x+t) - \log x \le t / x \quad \forall \, x, t > 0$$
it follows that the first rectangle has side length at least
$((1/3) (dr/r) / (1+dr/r)) / d\theta$. Since the collection is
nice, we have $dr / d\theta \ge 1/C_2$ and $r + dr \le \rho_2$,
so the side length is bounded below by $\delta_1 = 1/(3C_2 \rho_2)$.

Let $s = \log(1+dr/r) / d\theta$. Similarly we obtain
$$ 1 / (C_2 \rho_2) \le s \le 1 / (C_1 \rho_1), $$
and finally, the second rectangle must contain the rectangle
$$ [-\delta_2, s + \delta_2] \times [-\delta, 1+\delta] $$
for any $\delta_2<\delta\rho_1/(2C_2\rho_2^2)$.

We have reduced the original problem to the following: consider the rectangle
$R_2 = [0,s] \times [0,1]$ with $ 1 / (C_2 \rho_2) \le s \le 1 / (C_1 \rho_1) $.
Let $R_1$ be any rectangle of size $\delta_1 \times \eps$ contained inside $R_2$,
and let $R_3$ be the rectangle obtained by padding $R_2$ with bands of size
$\delta_2$ and $\delta$. The rectangles $R_2$ and $R_3$ have fixed or bounded size, and the diameter of $R$ is bounded from below, so
$$ \inf \pz ( T_{R_1} < T_{\partial R_3} ) = c_2 > 0 $$
where the infimum is taken over all $z \in R_2$, all $s$ and all possible positions
of $R_1$. This concludes the proof of~(\ref{eqc2}). Note that $c_2$ only depends
on $\rho_1, \rho_2, C_1$ and $C_2$. \\

\noi {\bf Step 4.}
The strong Markov property and equations~(\ref{eqc1}) and~(\ref{eqc2}) now yield
$$ \pz ( T_\la < T_{Q-Q_j} ) \ge c_2 c_1 \quad\forall z \in Q_j $$
and the proof is complete.
\end{proof}

We can now prove a generic form of the exponential bound stated at the beginning of this section, expressed in terms of a Riemann map.
\begin{Lemma}\label{tiles}
Let $D$ be a simply connected domain and let $\phi$ be a Riemann 
map that sends $D$ into the unit disk.
Let $W$ be a Brownian motion started at $x \in D$, conditioned to exit $D$ at $y \in \partial D$ and
killed upon exiting.

Let ${\cal Q} = \{Q_j\}$ be a nice collection of polar rectangles, so that each $Q_j$ touches $\partial D$ along one of its line sides, and the diameter $|D| \ge 2\,dr(Q_j)$ for all $j$.

Let $N$ be the number of rectangles in $\cal Q$ hit by $W$. Then $N$ has
exponential tails:
\begin{equation*}
\pxy(N \ge t) \le C \mu^t \quad \forall t \ge 1
\end{equation*}
where $C = 4 / (\mu d(\phi(Q), \phi(y))^2 \cdot d(\phi(Q), \phi(x)))$
and $\mu<1$ is a constant depending only on $\rho_1, \rho_2, C_1$ and $C_2$.
\end{Lemma}

\begin{proof}[Proof of Lemma \ref{tiles}]
We first consider a free Brownian motion started at $x$ and then we condition on exiting
at $y$. Recall that the Poisson kernel on the disk is
$$K(u,v) = (1-|u|^2) / |u-v|^2. $$
Consider a small boundary interval $I$ around $y$, 
such that its image under $\phi$ is an arc of length $dy$.
Clearly $K(\phi(x), v) \ge (1 - |\phi(x)|)/2$ for all $v$, so the probability that a free Brownian motion started at $x$
exits on $I$ satisfies
\begin{equation}\label{pxy_bottom}
\px (\wtd \in I) \ge (1-|\phi(x)|) \cdot dy/(4\pi).
\end{equation}

Let $Q = \cup Q_j$.
From Lemma~\ref{tiletile}, once the Brownian motion hits a rectangle, it is likely to hit 
the boundary before hitting another rectangle:
\begin{equation*}
\pz ( \tpd < T_{Q - Q_j} ) \ge c, \quad \forall j, \forall z \in Q_j.
\end{equation*}
Hence if $N$ is the number of rectangles hit before exiting $D$,
the strong Markov property gives the following bound:
\begin{equation*}
\px (N \ge t, \wtd \in I) \le \px(N \ge 1) \cdot (1-c)^{t-1} \cdot \sup_{z \in Q} \pz(\wtd \in I).
\end{equation*}
From Proposition~\ref{estunitdisk}, the first factor is bounded by
\begin{equation*}
\px(N \ge 1) \le 2 (1-|\phi(x)|) / d(\phi(Q), \phi(x))
\end{equation*}
while the last factor is easily bounded using the Poisson kernel:
\begin{equation*}
\sup_{z \in Q} \pz(\wtd \in I) \le dy / 2\pi d(\phi(Q),\phi( I))^2.
\end{equation*}
Hence
\begin{equation*}
\px (N \ge t \, | \, \wtd \in I) \le 4 \cdot (1-c)^{t-1} / (d(\phi(Q), \phi(I))^2 \cdot d(\phi(Q), \phi(x))),
\end{equation*} 
and letting $I$ converge to $\{ y \}$ completes the proof. Note that the proof works even if the rectangles are not subsets of $D$ (as long as they touch its boundary).
\end{proof}

The bound in Lemma~\ref{tiles} has certain nice features: the exponent $\mu$ depends only on the geometry of the collection of rectangles, but not on their number or size, or on the shape of the domain. However, the other constant $C$ depends on the Riemann map $\phi$, about which not much is known {\it a priori}. We will obtain an estimate for $\lambda$-domains in Proposition \ref{exp_tail_for_la_domain} by controlling the map $\phi$ and showing that $x$ and $y$ both map far enough from the elements of $\mcl Q$, see Figure \ref{fig-prop-app} for an illustration.

\begin{proof}[Proof of Proposition \ref{exp_tail_for_la_domain}]
Let $s = (1/2) (\eta_1 + \eta_2) \exp(-i \pi / 4)$,
 and let $\phi$ be a Riemann map that sends $D$ into the unit disk
and $s$ into the origin. We want to show that $\phi$ maps $x$ and $y$
far from where it maps the polar rectangles $Q_j$. The proof proceeds in several steps.\\

\noi {\bf Step 1.} We will now prove that for any $\rho,\rho'>0$ such that $\eta_1<\rho<\rho'<\eta_2$ there is a constant $C(\rho,\rho')$ such that the arc length of $\phi(\larr)$ is between $C(\rho,\rho')$ and $\pi$. We will first prove the upper bound. By conformal invariance,
the length of $\phi(\larr)$ is equal to $2 \pi$ times the probability 
of the event that a Brownian motion started at $s$ exits $D$ through $\larr$. In order to hit $\larr$ before
exiting, the Brownian motion started at $s$ has to hit the segment $[\eta_1, \eta_2]$ on the
$x$-axis before hitting $\sigma$. This event has probability 1/2, so
all $\phi(\larr)$ have arc length at most $\pi$.

To obtain a lower bound, note that for a Brownian motion started at $s$ to exit through $\larr$, it is enough if it exits the south-east quadrant through $[\rho,\rho']$ on the $x$ axis, and then moves in a counterclockwise direction inside the annulus $\{\rho \le |z| \le \rho'\}$ until it hits the negative $x$ axis. Hence if we take $C(\rho,\rho')$ to be $2\pi$ times the probability of the latter event, then $\phi(\larr)$ has arc length at least $C(\rho,\rho')$. We could obtain more precise estimates for $C(\rho,\rho')$, but all we will need is that it may depend on $\eta_1$ and $\eta_2$, but not on $\la$. \\

\noi {\bf Step 2.} If Brownian motion started at $z$ is likely to exit $D$ through $\larr$, then $\phi(z)$ must
be close to $\phi(\larr)$. This is clear intuitively, but we need a concrete estimate.
Let $S$ be any arc on the unit circle and
let $p(S)$ be the probability that a Brownian motion started at $u$ exits the unit disk {\bf outside} $S$.
Since the Poisson kernel satisfies
$K(u,v) = (1-|u|^2) / |u-v|^2 \ge (1 - |u|)/2$, we have
$$ p(S) \ge (1-|u|) \cdot l(S^c)/4\pi, $$
where $l(S^c)$ is the complement of $S$.
On the other hand, from Proposition~\ref{estunitdisk},
$$ 1-p(S) \le 2 (1-|u|) / d(u, S). $$
Hence $d(u,S) \le (8\pi / l(S^c)) \cdot p(S) / (1-p(S))$, and for $S=\phi ( \larr ) $ we obtain
\begin{equation}\label{plarr}
d(\phi(z), \phi(\larr)) \le 8 \pz ( \wtd \notin \larr) \, / \, \pz ( \wtd \in \larr).
\end{equation}
\\
\noi {\bf Step 3.} Fix some small $\eps > 0$ and let $L = \la[\rho_1 - \eps, \rho_2 + \eps]$. 
We show that all points of $\phi(Q_j)$ are close to $\phi(L)$ for all $j$. Let $z \in Q_j$,
and assume for now that $d\theta(Q_j) \le \pi / 4$, so the rectangle $Q_j$
remains far from the domain ``right'' boundary $\sigma$.
If a Brownian motion started at $z$ exits $D$ outside $L$, then it must either exit the annulus
$\{\rho_1 - \eps < |z| < \rho_2 + \eps \} $ (in which case it travels at least distance $\eps$)
or it must exit $D$ on $\sigma$ (in which case it travels at least distance $\eta_1 2^{-1/2}$).
Since $Q_j$ touches $\la[\rho_1, \rho_2] $,
the radial projection of $L$ with center $z$ contains the interval $( |Q_j|, \eps )$.
If we take $\eps \le \eta_1 2^{-1/2}$, then Beurling's theorem and Proposition~\ref{estBMsqrt}
give
$$ \pz ( \wtd \notin L) \le 2 \sqrt{ |Q_j| / \eps}. $$ 
Using conformal invariance and 
combining the last inequality with~(\ref{plarr}), we obtain that, if $|Q_j| \le \eps / 16$
and $d\theta(Q_j) \le \pi / 4$, then
\begin{equation*}
d( \phi(Q_j), \phi(L)) \le 32 \sqrt{ |Q_j| / \eps}.
\end{equation*}
Note that since the collection is nice, if we have $|Q_j| \le \pi / (4C_2)$, then this guarantees
that $d\theta(Q_j) \le \pi / 4$. \\

\noi {\bf Step 4.} We now show that $\phi(y)$ is far from $\phi(L)$. This is easy:
since $y \notin \la$, the arc distance between $\phi(y)$ and $\phi(L)$ must be at least
$ C(\eta_1, \rho_1 - \eps) \wedge C(\rho_2 + \eps, \eta_2) $. Hence
$$d(\phi(y), \phi(L)) \ge (2/\pi) C(\eta_1, \rho_1 - \eps) \wedge C(\rho_2 + \eps, \eta_2). $$

\noi {\bf Step 5.} It remains to show that $\phi(x)$ is far from $\phi(L)$.
Let $L' = \la[\rho_1 - 2\eps, \rho_2 + 2\eps]$ so $L'$ is a subset of $\la$ that is slightly
larger than $L$. We will show that the probability $p$ that a Brownian motion started at $x$ exits $D$ on
$L'$ is bounded above. 

If $\arg(x) \le -\pi / 4$, then clearly $p \le 1/2$.
Otherwise, we must have $|x| \le \rho_1 - \delta$ or $|x| \ge \rho_2 + \delta$.
Consider the first case and assume $\eps < \delta / 2$. See Figure \ref{fig-prop-app} for an illustration. Recall the definition of $u_{\rho_1-2\ep}$ from step 1, and let $\nu$ be the path contained in $\lambda$ which connects $\partial B(\eta_1)$ and $u_{\rho_1-2\ep}$. The path $\nu$ may divide $D\cap \{ \eta_1<|z|<\rho_1-\delta \}$ into several connected components, and observe that $x$ is in the component which intersects the positive real line by our assumption on $x$ in the statement of the proposition. The path $\lambda-\nu$ does not intersect $\nu$ by the assumptions on $\lambda$ in the definition of a $\lambda$-domain, and by definition of $u_{\rho_1-2\ep}$ the path $\lambda-\nu$ does not intersect the subset $L''$ of $\partial B(\rho_1-2\ep)\cap D$ which is to the right of $u_{\rho_1-2\ep}$. A Brownian motion started from $x$ which exits $D$ at $L'\subset \lambda-\nu$ must therefore cross $\partial B(\rho_1-2\ep)$ before it exits $D$. It follows that if a Brownian motion started at $x$ exits on $L'$, it hits the circle $\partial B(\rho_1-2\eps)$ before the circle $\partial B(\eta_1)$, so
$$ p \le \log(( \rho_1 - \delta ) / \eta_1 ) \, / \, \log(( \rho_1 - 2\eps ) / \eta_1 ). $$
A similar estimate holds in the second case, so we obtain that for all $x$,
$$ p \le K_1 \equiv K_1 ( \rho_1, \rho_2, \eta_1, \eta_2, \delta, \eps ) < 1. $$
Now, if a Brownian motion started at $\phi(x)$ does not exit on $\phi(L')$, then it must travel distance at least
$ d(\phi(x),\phi(L')^c)$ (where recall that $\phi(L')^c$ is the 
complement of $\phi(L')$ in the boundary of the unit disk).
Proposition~\ref{estunitdisk} gives
$$ 1 - K_1 \le 1-p \le 2 (1 - |\phi(x)|) \, / \, d(\phi(x),\phi(L')^c), $$
and this suffices to bound $ d(\phi(x),\phi(L)) $. Indeed, the last inequality gives
$$ d(\phi(x),\phi(L')^c) \le d(\phi(x),\phi(L)) \cdot 2 / (1-K_1), $$
but we also know that the arc distance between $\phi(L)$ and $\phi(L')^c$ is at least
$K_2 = C(\rho_1 - 2\eps, \rho_1 - \eps) \wedge C(\rho_2 + \eps, \rho_2 + 2\eps)$. Hence
by the triangle inequality,
$ d(\phi(x),\phi(L')^c) \ge (2/\pi) K_2 - d(\phi(x),\phi(L)) $, so finally
$$ d(\phi(x),\phi(L)) \ge (2/\pi) K_2 / (1 + 2/(1-K_1)). $$

\noi {\bf Step 6.}
Now all that is left is to combine the previous steps. Take $\eps = (1/3) \cdot 
\min(\rho_1 - \eta_1, \eta_2 - \rho_2, \delta, \eta_12^{-1/2})$.
Then $d(\phi(y),\phi(L)) \ge K_3$ and $d(\phi(x),\phi(L)) \ge K_3$ for some positive
constant $K_3$ that depends on $\eta_1, \eta_2, \rho_1, \rho_2$ and $\delta$.
If we also have
$$ |Q_j| \le K_4 := (\eps \cdot ((K_3  / 64)^2  \wedge ( 1 / 16)) ) 
\wedge (\pi / (4C_2)) \wedge (|D| / 2)$$
then $d(\wt x,\phi(L)) \le K_3 / 2$ for all $\wt x\in \phi(Q_j)$. Then from the triangle inequality, $d(\phi(Q_j),\phi(y)) \ge K_3 / 2$ and $d(\phi(Q_j),\phi(x)) \ge K_3 / 2$.

Let $N$ be the total number of rectangles hit. We can write $N = N_1 + N_2$, where
$N_1$ is the number of rectangles with diameter smaller than $K_4$ that are hit, and
$N_2$ is the number of rectangles with diameter larger than $K_4$ that are hit. Since
the collection $\cal Q$ is nice, all its sub-collections are also nice, so Lemma~\ref{tiles} gives 
$$ \pxy(N_1 \ge t) \le (32 / K_3^3\mu) \cdot \mu^t. $$
To estimate $N_2$, observe that if $Q_j$ has diameter larger than $K_4$, then
$$ K_4 \le |Q_j| \le dr(Q_j) + \rho_2 \cdot d\theta(Q_j) \le (1+C_2 \rho_2) \cdot dr(Q_j) $$
so the area of $Q_j$ is bounded below by
$$ \leb{Q_j} \ge \rho_1 \cdot dr(Q_j) \cdot d\theta(Q_j) 
\ge (\rho_1 C_1) ((K_4 / (1+C_2 \rho_2) )^2 \equiv K_5. $$
Since all rectangles $Q_j$ are disjoint and inside the annulus $ \{ \rho_1 \le |z| \le \rho_2 \} $, 
there are at most $K_6 = \pi (\rho_2^2 - \rho_1^2) / K_5 $ rectangles with diameter larger than $K_4$.
So $N_2 \le K_6$, and
$$ \pxy(N \ge t) \le\pxy(N_1 \ge t - K_6) \le (16 / K_3^3) \cdot \mu^{-K_6-1} \cdot \mu^t. $$
\end{proof}

\bibliography{brownian}

\end{document}